\documentclass[11 pt, twoside ]{article}
\usepackage{amssymb,amsmath,latexsym,amsthm}
\usepackage{graphicx}
\usepackage{textcomp}
\usepackage[]{hyperref}
\usepackage{amsmath,amscd}
\usepackage[all,cmtip]{xy}
\usepackage{microtype}
\usepackage[german,english]{babel}
\usepackage[total={6.5in, 9.000in}, top=1.2in, left=0.9in, includefoot]{geometry}
\usepackage{setspace}
\usepackage{epsfig}
\input{xy}
\numberwithin{equation}{section}
\def\refname{\centerline{Bibliography}}
\newtheorem{theorem}{Theorem}[section]
\newtheorem{definition}[theorem]{Definition}
\newtheorem{lemma}[theorem]{Lemma}
\newtheorem{proposition}[theorem]{Proposition}
\newtheorem{corollary}[theorem]{Corollary}

\theoremstyle{definition}{}
\newtheorem{remark}[theorem]{Remark}

\newtheorem{example}[theorem]{Example}
\numberwithin{equation}{section}

\def\refname{\centerline{Bibliography}}

\title{Weight reduction for cohomological  mod $p$ modular forms over imaginary   quadratic fields}

\author{Adam Mohamed\thanks{Universit\"at Duisburg-Essen, Institut f\"ur Experimentelle Mathematik, Ellernstr 29, 45326 Essen, Germany,  
E-mail: \url{adam.mohamed@uni-due.de}}}

\begin{document}
\maketitle
\begin{abstract}
Let $ F$ be an imaginary quadratic field and  $ \mathcal{O}$  its ring of integers . Let $ \mathfrak{ n} \subset \mathcal{ O} $ be a non-zero ideal and  let $ p> 5$ be a rational inert prime  in $F$ and coprime with $ \mathfrak{ n}.$ Let $ V$ be an irreducible finite dimensional representation of $ \overline{\mathbb{F}}_{p}[{\rm GL}_2(\mathbb{F}_{ p^2})].$
We establish that  a  system of Hecke eigenvalues appearing in the cohomology with coefficients in $ V$ already lives in the cohomology with coefficients in $ \overline{\mathbb{F}}_{p}\otimes det^e$ for some  $ e \geq 0$; except possibly in some few cases.
\end{abstract}
\bigskip
\section{Introduction} 
Let  $ F$ be an imaginary quadratic field with $ \mathcal{ O}$ as its  ring of integers. The class number of $ F$ is denoted as $h.$ Let  $\Gamma$  be a congruence subgroup of $ {\rm GL}_2( \mathcal{ O}).$ 
Let $\sigma$ be the non-trivial element of  ${\rm Gal}( F/ \mathbb{ Q}).$ We consider the representations of $ {\rm GL }_2( \mathcal{ O})$ defined as $ V^{a, b}_{r, s}(\mathcal{O}) =  {\rm Sym }^{r}( \mathcal{O}^2)\otimes det^a \otimes ({\rm  Sym }^s( \mathcal{O}^2))^{\sigma}\otimes (det^b)^\sigma$ where $ a, b ,  r,  s$ are  positive integers. For an  $ \mathcal{O}$-algebra $ A,$ we define $V^{a, b}_{r, s}(A)  := V^{a, b}_{r, s}(\mathcal{O})\otimes_\mathcal{O} A.$
A cohomological modular form of level  $\Gamma$ and weight $ V^{a, b}_{r, s}(A)  $ over $F$ is a class in  $ {\rm  H}^1(\Gamma,  V^{a, b}_{r, s}( A )).$ As in the classical setting, the space $ {\rm  H}^1( \Gamma,  V^{a, b }_{r, s}( A)) $ can be endowed with a structure of Hecke module. The Hecke algebra acting on $ {\rm  H}^1( \Gamma,  V^{a, b }_{r, s}(A)) $ is commutative and has its  elements  indexed over the integral ideals  of $ F.$ So, one can consider eigenclasses\index{eigenclasses} or eigenforms\index{eigenforms} which are eigenvectors for all the Hecke operators  $ T_\mathfrak{a}\index{ $T_\mathfrak{a},\, \text{Hecke operator}$}.$ Hence to such an eigenform corresponds a system of Hecke eigenvalues.

Integral systems of eigenvalues when reduced modulo a prime $ p$ are believed to be related to mod $p$ representations of  Galois groups  as conjectured by Ash et al. in \cite{Ash1}. One instance  of this correspondence being the theorem of Deligne constructing   $l$-adic representations of the absolute Galois group of $ \mathbb{ Q}, $ $ G_\mathbb{ Q} :=  Gal( \overline{\mathbb{Q}}/\mathbb{ Q}), $ via systems of Hecke eigenvalues arising from modular forms over  $ \mathbb{ Q}$. Let $ N$ be a positive integer and  $ \Gamma_0( N)$  a congruence subgroup of $ {\rm SL }_2(  \mathbb{Z}).$ Take $ V$  to be  the $ {\rm  SL}_2( \mathbb{ Z})$-module given as $V := Sym^{k-2}( \mathbb{Z}^2) = \mathbb{Z}[ X, Y]_{k -2},$ the space of homogeneous polynomials of degree $k-2$ over $ \mathbb{Z}$ in two variables  and with $ k $ even. The converse of Deligne's theorem,  Serre's modularity conjecture, which is now a theorem of Khare and Wintenberger,  has been formulated in the language of group cohomology  in \cite{AshW} and the standard conjecture in there relates mod $p$ Galois representations of $ G_\mathbb{ Q}$ to systems of Hecke eigenvalues on $ {\rm  H}^1(\Gamma_0(N), V\otimes_\mathbb{Z}\overline{\mathbb{F}}_p).$  

Next let $N$ and $n$ be  positive integers. In \cite{Ash-Stevens}, it was shown that a system of Hecke eigenvalues occurring in the cohomology of $\Gamma_1( N)$ with coefficients in some $ {\rm GL}_n( \mathbb{ F}_p)$-module also occurs in the cohomology with coefficients in some irreducible $ {\rm GL}_n(\mathbb{ F}_p)$-module.  This fact has some interesting features. In fact it allows one to obtain  a cohomological avatar of the so-called Hasse invariant, see \cite{Edix}. That is,  one can produce congruences between  weight two and higher weight modular forms using cohomological methods.

As for the case of an imaginary quadratic field $F$ of class number one,  then when $ p$ splits in $ F$ and is  coprime with $\mathfrak{n},$ in \cite{Haluk-Seify},  it is established that a Hecke system of eigenvalues occurring in the first cohomology with non-trivial coefficients can be realized in the first cohomology with trivial coefficients. This  should also hold when the class number of $F$ is greater than one.

Let  $ p$ be a rational prime  coprime to $ \mathfrak{ n}$ and inert in $ F.$ Let $ E$ be a finite dimensional representation of ${\rm GL_2(\mathbb{F}_{p^2})}  $ over $ \overline{\mathbb{F}}_p.$ Let $ \Gamma$ be a congruence subgroup of $ {\rm GL_2( \mathcal{ O})}.$ Then a cohomological  mod $ p$ modular form of level $ \Gamma$ and weight $ E$ is defined to be a class in $\rm{H}^1 ( \Gamma, E ).$  As in the classical setting there is a Hecke algebra action on the space $\rm{H}^1(\Gamma, E )$ and  one can consider  systems of Hecke eigenvalues for the space $\rm{H}^1 ( \Gamma, E )$.
Our  aim  will be to say something more precise about systems of Hecke eigenvalues\index{systems of Hecke eigenvalues} in this setting.  We will prove that a system of Hecke eigenvalues living in $ \oplus^h_{i = 1}{\rm  H^1} (\Gamma_{1, [\mathfrak{b}_i]}(\mathfrak{ n} ), M)$ where $ M$ is an irreducible $\overline{\mathbb{ F}}_p[ {\rm GL_2( \mathbb{ F}_{ p^2} )}]$-module also  occurs in $ \oplus^h_{ i = 1}{\rm H}^1( \Gamma_{1, [\mathfrak{b}_i]}(\mathfrak{pn}), \overline{\mathbb{F}}_p\otimes det^e)$ for some $ e \geq 0$ depending on $M;$ except possibly for some cases. See Theorem  \ref{RedThm} for the precise statement.  Here $ \Gamma_{ 1, [\mathfrak{b}_i]}(\mathfrak{n})\index{$\Gamma_{ 1, [\mathfrak{b}_i]}(\mathfrak{n})$, \,\text{congruence subgroups}}$ are some congruence subgroups defined in Section \ref{sec2_3}.

There is an application of Theorem \ref{RedThm} related to  Serre type questions  about mod $p$ Galois representations of the absolute Galois group of $ F.$ When we are dealing with cohomological modular forms mod $p$ with  trivial coefficients $ \overline{\mathbb{F}}_p, $ we shall say that we are in {\it  weight two}\index{weight two}. Let $  G_F :=  \textrm{Gal}( \overline{F}/ F)$ and  let be given 
\begin{displaymath} \rho: G_F \rightarrow {\rm GL}_2(\overline{\mathbb{F}}_p)                                       \end{displaymath}
an irreducible mod $p$ Galois representation\index{Galois representation} of conductor $\mathfrak{n}$.
Let $ Tr$ denotes the trace of a matrix. Then the following questions arise:
\begin{enumerate}\label{Serre'squestions}
\item[(a)] Does there  exist a cohomological Hecke eigenform of some weight  $ V$ and  level $\mathfrak{n}$ with eigenvalues $ \psi(T_\lambda)$ such that  $  Tr( \rho( Frob_\lambda)) = \psi( T_\lambda)$ for all unramified prime ideals $ \lambda \nmid  \mathfrak{ pn}?$ 
\item[(b)] Does there exist  a cohomological Hecke eigenform of weight $ 2$ ($  V = \overline{\mathbb{ F}}_p\otimes det^e $   for some $ e\geq 0$) and level $\mathfrak{pn} $  with $ Tr(\rho( Frob_\lambda)) = \psi( T_\lambda)$ for all unramified prime ideals $ \lambda \nmid  \mathfrak{ pn}?$
\end{enumerate}
As a consequence of Theorem \ref{RedThm}, we shall see that the two questions above are equivalent. See Proposition \ref{SerreA} for the precise statement.
Proposition \ref{SerreA} proves that when investigating Serre type questions as above, it is enough to work in weight two. For example, in \cite{Fi},  some computational investigations of Serre's  conjecture over imaginary quadratic fields were carried out  and the principle illustrated by Proposition \ref{SerreA} was assumed to hold.

Here is our outline. We shall first recall Hecke theory in our context. This is the content of Section \ref{sect2_2}. In Section \ref{sec2_3}, we shall compare some modules.  The main result is proved in Section \ref{sec2_4}.
\subsubsection*{Acknowledgement}
The present article is extracted from the author dissertation which was supervised by Gabor Wiese. The author thanks him for his time and teaching. This work started as one of the FP$6$ European Research Training Networks  ``Galois Theory and Explicit Methods'' project (GTEM; MRTN-CT-2006-035495 ), I  acknowledge their financial support. 
\section{Hecke operators}\label{sect2_2}
We define Hecke operators\index{Hecke! operators} via Hecke correspondences\index{Hecke! correspondences} on hyperbolic $3$-manifolds. 
We start by fixing some notation.
Let  $ F$ be an imaginary quadratic field of class number $ h\geq 1 $. Denote by $ \mathcal{ O}$ its ring of integer and let $ \mathfrak{ n}$ be  an ideal of $ \mathcal{O}.$ The class group of  $F$ is denoted by $ Cl$ and we fix a rational prime $ p$ inert in $F$ and $\mathfrak{p}  = p\mathcal{O}.$ We also assume that $ \mathfrak{ p}$  is  coprime with $\mathfrak{n}.$
Let  $\hat{\mathcal{O}}$ be the profinite completion of $\mathcal{O}: \hat{\mathcal{O}} = \prod_{\mathfrak{q}\neq 0}\mathcal{O}_\mathfrak{q}.$   We will denote the adeles of $ F$ by $ \mathbb{A}, $ and $ \mathbb{ A}_f, \, \mathbb{ A}_\infty$ stand for the finite part and the infinite part of $ \mathbb{A}.$ We write $ G : =  \rm GL_2,$  so that,   $  G(\mathbb{ A}),  G ( F), G(\mathbb{A}_f) $ are the usual linear algebraic groups of  $ 2\times2$ matrices with entries in  $\mathbb{ A},  F, \mathbb{A}_f,  $ respectively. Let  $ \mathbb{ H}_3 :=   G(  \mathbb{  C})/ \mathbb{ C}^*{\rm  U}_2 \cong  \mathbb{C}\times \mathbb{ R}_{ > 0},$ the three dimensional equivalent of the classical Poincar\'{e} upper half plane $ \mathbb{ H}_2 =  G(\mathbb{ R } )/\mathbb{ R}^*{\rm O}_2$.  Here  ${\rm U}_2 $ is the unitary subgroup of $G(\mathbb{C}).$ Let $ K\index{$K$,\,\text{open compact subgroup}}$ be  an open compact subgroup of  $ G(\hat{\mathcal{O}})$ such that  the determinant homomorphism  $$ det: K \rightarrow \hat{\mathcal{ O}}^* $$ is surjective. We  define the following homogeneous space  
\begin{eqnarray*} 
Y_K\index{$Y_K$}  &:= &  G( F)\backslash ( \mathbb{ H}_3\times G( \mathbb{ A}_f)/K)  \\
& = & G( F)\backslash (G(\mathbb{C})/ \mathbb{ C}^*{\rm  U}_2\times G(\mathbb{A}_f )/K) \\  
& = & G( F)\backslash G(\mathbb{A})/ K.U_2. \mathbb{C}^*.
\end{eqnarray*}
By the determinant map we have   $$ Y_K  \twoheadrightarrow F^*\backslash\mathbb{ A}^*/ \hat{\mathcal{ O}}^* \mathbb{ C}^*\cong F^*\backslash\mathbb{ A}^*_f / \hat {\mathcal { O}}^* \cong  Cl.$$
\subsection{Hecke correspondences and Hecke operators}\label{Subsec2_2_1}
Here we shall recall how a sheaf of $ \overline{\mathbb{F}}_p$-modules  on  $ Y_K$ associated to a finite dimensional representation of $ \overline{\mathbb{F}}_p[G(\mathbb{F}_{p^2})]$ is constructed. 
So, let  $\sigma$ be the generator of ${\rm Gal}(F/\mathbb{Q}).$  Let $$ V_{\mathcal{O}} = V^{a, b}_{r, s}(\mathcal{O})\index{$V^{a, b}_{r, s}(\mathcal{O})$} =  {\rm Sym}^{r}(\mathcal{O}^2)\otimes det^a \otimes ({\rm  Sym }^s(\mathcal{O}^2))^{\sigma}\otimes (det^b)^\sigma$$ be an $ \mathcal{O}[G( \mathcal{O})]$-module endowed with the discrete topology. We define $ V^{a, b}_{r, s}( \overline{\mathbb{F}}_p)\index{$V^{a, b}_{r, s}( \overline{\mathbb{F}}_p)$}  : =  V_{\mathcal{O}} \otimes_\mathcal{O} \overline{\mathbb{F}}_p.$
This space  is also endowed with the discrete topology.

On the space  $\mathbb{H}_3 \times G(\mathbb{A}_f)\times V^{a, b}_{r, s}( \overline{\mathbb{F}}_p),$  the  group  $G(F)$ acts on the first two factors from the left and the group $K$ acts on the last two factors from the right. We write these double actions as follows. Let $ (q,  k) \in $   $ G(F)\times K$ and $ (h, g,  v) \in \mathbb{H}_3 \times G(\mathbb{A}_f)\times V^{a, b}_{r, s}( \overline{\mathbb{F}}_p)$  then: $$  (q,  k)*( h, g,  v) := (qh, qgk, k^{-1}.v).$$ 
Taking the quotients of  these actions of $ G(F), \, K$ on $ \mathbb{ H}_3\times G( \mathbb{ A}_f)\times V^{a, b}_{r, s}( \overline{\mathbb{F}}_p)$  yields a locally constant sheaf $ \mathcal{V}_{\overline{\mathbb{ F}}_p}$ of $ \overline{\mathbb{ F}}_p$-vector spaces associated to $V^{a, b}_{r, s}( \overline{\mathbb{F}}_p)$ on  $Y_K.$ More precisely let $ X = G(\mathbb{A})/  U_2 \mathbb{C}^* \cong  \mathbb{ H}_3 \times G(\mathbb{A}_f).$ Under the  assumption that  $K$ acts freely on $X\times V^{a, b}_{r, s}( \overline{\mathbb{F}}_p), $ one has a topological cover  $$ \pi_1: G(F )\backslash (X\times V^{a, b}_{r, s}( \overline{\mathbb{F}}_p))/K  \rightarrow G( F)\backslash X/ K \cong Y_K.$$ And the locally constant  sheaf  $\mathcal{V}_{\overline{\mathbb{ F}}_p}$ on $ Y_K$ is given by the sections of $\pi_1:$ for an open subset $ U $ of $ Y_K,$ we have  $$ \mathcal{V}_{\overline{\mathbb{ F}}_p}( U) = \{ s: U \rightarrow G(F )\backslash (X\times V^{a, b}_{r, s}( \overline{\mathbb{F}}_p))/K; \,\, \pi_1\circ s = id\}.$$
Let $ K'\subset K$ be  another compact open subgroup of  $G(\mathbb{A}_f).$  We have the natural projection $ \psi: Y_{K'} \rightarrow  Y_{K}.$
We define the locally constant sheaf $ \psi^{ -1}\mathcal{V}_{\overline{\mathbb{ F}}_p}$  of  $\overline{\mathbb{F}}_p$-vector spaces on $ Y_{K'}$ as the pull back of the sheaf $ \mathcal{V}_{\overline{\mathbb{ F}}_p}.$  
By functorial properties of sheaf cohomology the map $ \psi$ induces a homomorphism of $\overline{\mathbb{F}}_p$-vector spaces similar to the restriction homomorphism\index{restriction homomorphism} in group cohomology:
$$ res: {\rm H}^r(Y_K,  \mathcal{V}_{\overline{\mathbb{ F}}_p})\rightarrow {\rm H }^r( Y_{K'},  \psi^{-1}\mathcal{V}_{\overline{\mathbb{ F}}_p}).$$
Since $ K'  $ is a subgroup of finite index inside  $ K,$ we have available the transfer map also known as the corestriction map\index{corestriction map} : $$ cor\index{$cor$, \,\text{corestriction map}}: {\rm H }^r( Y_{K'}, \psi^{-1}\mathcal{V}_{\overline{\mathbb{ F}}_p}) \rightarrow {\rm  H}^r( Y_K, \mathcal{V}_{\overline{\mathbb{ F}}_p}).$$

Next let $ g \in   Mat_2( \hat{\mathcal{O}})_{\neq 0}$ be such that all its local factors $  g_\mathfrak{q}$ at almost all the finite  places $ \mathfrak{q}$ including those  dividing $ \mathfrak{pn}$ are $ \left(\begin{smallmatrix} 1 & 0 \\0 & 1 \end{smallmatrix}\right)$ and otherwise $ g_\mathfrak{q}$ are of the form $  \left(\begin{smallmatrix} \pi_\mathfrak{q} & 0 \\0 & 1 \end{smallmatrix}\right)$ or  $ \left(\begin{smallmatrix} \pi^2_\mathfrak{q} & 0 \\0 & 1 \end{smallmatrix}\right)$ with $ \pi_\mathfrak{q}$ a uniformizer of $ \mathcal{O}_\mathfrak{q}.$ 
\begin{remark}
Often one takes $ g\in Mat_2(\hat{\mathcal{O}})$ with the component at only one  finite place $\mathfrak{q}$ away from  $\mathfrak{pn}$ $ g_\mathfrak{q}$ being of the form $ \left(\begin{smallmatrix} \pi_\mathfrak{q} & 0 \\0 & 1 \end{smallmatrix}\right)$ and all the remaining components are the identity matrices.
\end{remark}
We define   $K'_{ g^{-1}} = K \cap g^{-1}Kg$ and $ K'_{ g} = gK g^{-1} \cap K.$  The group isomorphism $$  K'_{ g^{-1}} \cong  K'_{ g} ; \,\, \lambda \mapsto   g\lambda g^{-1}$$  induces the isomorphism  $  g^* : Y_{K'_{ g^{-1}}} \cong  Y_{K'_{g}};  y \mapsto g y.$ We can now form the diagram
$$\begin{CD}
Y_{K'_{ g^{-1}}} @> g^*>>   Y_{K'_{ g}}   \\
@VVs_gV   @VV\tilde{s}_gV   \\
Y_{K}  @. Y_K, 
\end{CD}
$$
where $ s_g$ and $\tilde{s}_g$ are the natural projections.
This  diagram is called a {\it Hecke correspondence}\index{Hecke! correspondence} in light of the classical Hecke correspondence for modular curves. This picture is the essence of Hecke operators on  cohomology,  the notion of which we shall recall the definition in  a moment.
We denote by  $ s^{-1}_g\mathcal{V}_{\overline{\mathbb{F}}_p}, \,\,\tilde{s}^{-1}_g\mathcal{V}_{\overline{\mathbb{F}}_p}$ the sheaves on $ Y_{K'_{ g^{-1}}}$ and $  Y_{ K'_g}$ respectively, obtained as the pull back of the sheaf $\mathcal{V}_{\overline{\mathbb{F}}_p}$ of $\overline{\mathbb{F}}_p$-vector spaces on  $ Y_K.$
Note that we have an isomorphism of sheaves induced by $ g^*, \,\, conj_g:  \tilde{s}^{-1}_g\mathcal{V}_{\overline{\mathbb{F}}_p} \cong s^{-1}_g\mathcal{V}_{\overline{\mathbb{F}}_p}.$ Hence an isomorphism on cohomology 
$$ conj^*_g: {\rm  H}^i( Y_{ K'_{g^{ -1}}} ,s^{-1}_g\mathcal{V}_{\overline{\mathbb{ F}}_p})    \cong {\rm  H}^i ( Y_{K'_{g}}, \tilde{s}^{-1}_g\mathcal{V}_{\overline{\mathbb{ F}}_p})$$ holds.
The Hecke operator $ T_g\index{$T_g$,\,\text{Hecke operator}}$ acting on the $\overline{\mathbb{F}}_p$-vector spaces  ${\rm  H}^i(Y_K, \mathcal{V}_{\overline{\mathbb{ F}}_p})$ is defined by the  following diagram: $$ \begin{CD}  {\rm  H}^i( Y_{ K'_{g^{ -1}}}, s^{-1}_g\mathcal{V}_{\overline{\mathbb{ F}}_p}) @>conj^*_g>> {\rm  H}^i(Y_{K'_g}, \tilde{s}^{-1}_g\mathcal{V}_{\overline{\mathbb{ F}}_p}) \\
@AAresA   @VVcorV  \\
{\rm H}^i(Y_K, \mathcal{V}_{\overline{\mathbb{ F}}_p}) @. {\rm H}^i( Y_K, \mathcal{V}_{\overline{\mathbb{ F}}_p}).
\end{CD}
$$
So we have  $ T_g = cor\circ conj^*_g\circ res.$  It is also known that  $ T_g$ is independent of the choice of the uniformizers $ \pi_\mathfrak{q}$  but in fact depends only on the double coset $KgK\index{$ KgK$,\,\text{double coset}}.$ When $ g\in Mat_2( \hat{\mathcal{O}})$ has local components 
$ \left(\begin{smallmatrix} \pi^2_\mathfrak{q} & 0 \\ 0 &  1\end{smallmatrix}\right) $  at a finite  number of  finite places $ \mathfrak{q}$ away from $ \mathfrak{pn}$ and  the identity otherwise, we shall denote the corresponding Hecke operator as $ S_g\index{$S_g$,\, \text{Hecke operator}}$. For the  sake of our understanding, let us translate the above diagram in group cohomology  and have a more explicit description of the Hecke operator $ T_g.$
\subsection{Comparison with group cohomology and Hecke algebra}\label{Subsection2_2_2}
Let $ \mathfrak{n}$ be a  non-zero ideal of $\mathcal{O}.$
For our purposes, we choose the following representatives of the class group $ Cl\index{$Cl$,\, \text{class group}}$ of $ F.$ By the Chebotarev density theorem, we can choose representatives of the class group    $ [\mathfrak{b}_1] = [ \mathcal{O}],  [\mathfrak{b}_2], \cdots, [\mathfrak{b}_h],$ where for $  i > 1,$  the  $  \mathfrak{b}_i$ are prime ideals coprime with $ \mathfrak{pn}.$
Thus we denote  the class group as $ Cl = \{[\mathfrak{b}_1], \cdots, [\mathfrak{ b}_h]\}.$ Let $ \pi_{\mathfrak{b}_i}$ be a uniformizer of the local ring $ \mathcal{O}_{\mathfrak{b}_i}.$  We define  $ t_1 := ( 1,\cdots,1, 1, 1,\cdots, 1, \cdots) $,  and for $ i > 1, $   $  t_i := (1,\cdots,1, \pi_{\mathfrak{b}_i}, 1,\cdots, 1, \cdots) \in \mathbb{ A}^*_f$, i.e,  $ t_i$ is the idele having $ 1$ at all places expect at the place $ \mathfrak{b}_i$ where we have $\pi_{\mathfrak{b}_i}.$
Via the group homomorphism
\begin{eqnarray*}
\mathbb{ A}^*_f & \rightarrow & Cl \\  (\cdots x_\mathfrak{ q} \cdots) &\mapsto& [ \prod_{ \mathfrak{q}\neq \infty} \mathfrak{ q}^{ v_\mathfrak{ q}( x_\mathfrak{ q})}],
\end{eqnarray*} 
where $ v_\mathfrak{q}$ is the normalized valuation of $ \mathcal{O}_\mathfrak{q}, $ we see that $ t_i$ corresponds to $ \mathfrak{b}_i.$
We define  $g_i  := \left(\begin{smallmatrix} t_i& 0 \\ 0 & 1 \end{smallmatrix}\right),$ i.e,  $ (g_i)_\mathfrak{q} = \left(\begin{smallmatrix} (t_i)_\mathfrak{q}& 0 \\ 0 & 1 \end{smallmatrix}\right).$ Similarly $ g_i$ corresponds to the  class $ [ \mathfrak{b}_i]$ via the determinant map.

From the strong approximation theorem, the topological space $ Y_K$ decomposes into the disjoint union of its connected components as: $$  Y_K =  \amalg^h_{i = 1} \Gamma_{[\mathfrak{b}_i]}\backslash \mathbb{H}_3,$$ where $ \Gamma_{[\mathfrak{b}_i]}\index{$\Gamma_{[\mathfrak{b}_i]}$} :=  G(F) \cap g_iK g^{-1}_i$. This is an arithmetic subgroup of $ G( F).$ We next recall the definition of neatness for subgroups of $ G(\mathbb{A}_f).$ This is the condition to ensure that $ K $ acts freely on $ X\times V^{a, b}_{r, s}( \overline{\mathbb{F}}_p),$ where $ X$ and  $ V^{a, b}_{r, s}( \overline{\mathbb{F}}_p)$ are defined in Subsection \ref{Subsec2_2_1}.
\subsubsection{Neatness}
Let  $\overline{F}$ be an algebraic closure of $F.$ A subgroup $ \Gamma$ of  $ G(F)$ is said to be  {\it neat }\index{neat} if and only if for all $ g \in \Gamma,$ the multiplicative subgroup of  $\overline{F}^*$ generated by all the eigenvalues of $ g$ is torsion free. If $ \Gamma$ is neat then it is torsion free. Let  $ \mathfrak{q}$ be a finite place of $ F$ and consider $ F_\mathfrak{q}.$ We fix an embedding $ \overline{F} \hookrightarrow \overline{F}_\mathfrak{q}.$ Let  $ g = ( g_\mathfrak{q}) \in  G(\mathbb{A}_f)$ and let  $ \Omega_\mathfrak{q}$ be the subgroup of $ \overline{F}^*_\mathfrak{q}$ generated by all the eigenvalues of $ g_\mathfrak{q}.$ One says that $  g$ is neat if only if we have $$ \bigcap_\mathfrak{q}( \overline{F}^*\cap \Omega_\mathfrak{q})_{tor} = \{1\}.$$  A subgroup $ \Omega$ of $ G(\mathbb{A}_f)$ is neat if and only if all its elements are neat.  For  $ g \in G(\hat{\mathcal{O}}),$  if the open compact subgroup   $ K$ of $ G(\hat{\mathcal{O}}) $ is neat then  $G( F)\cap g K g^{-1} $ is also neat. For more on the neatness condition see Borel \cite[p. 117]{Borel}.

So we choose $K$ to be neat so that the  groups $ \Gamma_{[\mathfrak{b}_i]}$ are torsion free. To achieve this, if   $ K = K_1( \mathfrak{n} ),$ the open compact subgroup of level $ \mathfrak{n}$ defined below,  where  the positive  generator of $ \mathfrak{n}\cap \mathbb{Z}$ is greater than $ 3,$ then $  \Gamma_{[\mathfrak{b}_i]}$ are torsion free. This is Lemma 2.3.1 from \cite{Urban}. This being given, the smooth manifolds 
$ \Gamma_{[\mathfrak{b}_i]}\backslash\mathbb{H}_3$ are Eilenberg-McLane spaces\index{Eilenberg-McLane spaces} of type $ K(\Gamma_{[\mathfrak{b}_i]},  1)$ ( this  $ K$ has nothing to do with our open compact subgroup $K $, this is just an unfortunate clash between two pieces of standard notation),  i.e,  $ \pi_1( \Gamma_{[\mathfrak{b}_i]}\backslash\mathbb{H}_3) = \Gamma_{[\mathfrak{b}_i]}$ and $ \pi_n( \Gamma_{[\mathfrak{b}_i]}\backslash\mathbb{H}_3) = 1  $ for $ n> 1.$ 
From  a general comparison  theorem it is known that an isomorphism ${\rm  H}^r(\Gamma_{[\mathfrak{b}_i]}\backslash \mathbb{H}_3,  \mathcal{V}_{\overline{\mathbb{ F}}_p}) = {\rm H}^r( \Gamma_{[\mathfrak{b}_i]}, V^{a, b}_{r, s}( \overline{\mathbb{F}}_p)) $ holds, see \cite{Brown} for details.  Hence we can  write   $$ {\rm  H}^r (Y_K, \mathcal{V}_{\overline{\mathbb{ F}}_p}) = \oplus^h_{i = 1}{\rm H }^r ( \Gamma_{[\mathfrak{b}_i]}\backslash \mathbb{ H}_3,  \mathcal{V}_{\overline{\mathbb{ F}}_p})=\oplus^h_{i = 1} {\rm H }^r ( \Gamma_{[\mathfrak{b}_i]},  V^{a, b}_{r, s}( \overline{\mathbb{F}}_p)).$$ 

Let us further specialize the open compact subgroup $ K.$ 
We define the open compact subgroup of level  $ \mathfrak{n}$ $$ K_1(\mathfrak{n}) = \big\{  \left(\begin{smallmatrix} a & b \\ c & d \end{smallmatrix}\right) \in \prod_{ \mathfrak{ q}\nmid \infty} G( \mathcal{O}_\mathfrak{q}): c, d-1   \in \mathfrak{n}\hat{\mathcal{ O}}\big\}.$$
This is an open compact subgroup which surjects on  $ \hat{\mathcal{O}}^*$ by the determinant map.   The corresponding congruence subgroups  $ G(F) \cap g_i K_1( \mathfrak{ n}) g^{-1}_i $ are denoted as $\Gamma_{1, [\mathfrak{b}_i]}(\mathfrak{n}).$
As already alluded to,  the Hecke operators $ T_g$ do not  act componentwise on the $ \overline{\mathbb{F}}_p$-vector space $ \oplus^h_{i = 1}{\rm H}^r(\Gamma_{1, [\mathfrak{b}_i]},  V^{a, b}_{r, s}( \overline{\mathbb{F}}_p)).$ By this  we mean that  in general $ T_g$ permutes the components when acting on an element from $\oplus^h_{i = 1}{\rm H}^r(\Gamma_{1, [\mathfrak{b}_i]},  V^{a, b}_{r, s}( \overline{\mathbb{F}}_p))$ as we will soon see.
\subsubsection{Some formulas for the Hecke action}\label{SubsecH}
We recall here the formulas defining the Hecke action on group cohomology. To this end, let us first introduce some more notation.
Let  $\mathfrak{q}$ be an integral ideal away from $\mathfrak{pn}.$ We consider the following subset of $ Mat_2(\hat{\mathcal{O}}).$ Define $$\Delta^\mathfrak{q}_1(\mathfrak{n}) = \{\left(\begin{smallmatrix} a & b \\ c & d \end{smallmatrix}\right)\in Mat_2(\hat{\mathcal{O}}): (ad-bc )\hat{\mathcal{O}} = \mathfrak{q}\hat{\mathcal{O}},\left(\begin{smallmatrix} a & b \\ c & d \end{smallmatrix}\right) \equiv \left(\begin{smallmatrix} * & * \\ 0 & 1 \end{smallmatrix}\right) \pmod{\mathfrak{n}} 
\}.$$
The open compact subgroup  $K_1(\mathfrak{n})\index{$K_1(\mathfrak{n})$}$ acts on  $\Delta^\mathfrak{q}_1(\mathfrak{n})\index{$\Delta^\mathfrak{q}_1(\mathfrak{n})$}$ via multiplication: for $ g \in K_1( \mathfrak{n})$ and $ \delta \in \Delta^\mathfrak{q}_1(\mathfrak{ n} ) $ we have $ g\delta \in \Delta^\mathfrak{q}_1(\mathfrak{ n}).$  We have that $\Delta^\mathfrak{q}_1(\mathfrak{ n} )K_1(\mathfrak{n}) = K_1(\mathfrak{n})\Delta^\mathfrak{q}_1(\mathfrak{n}) = \Delta^\mathfrak{q}_1(\mathfrak{n}).$ 
For $\delta \in \Delta^\mathfrak{q}_1(\mathfrak{n})$ we define the subgroup $$ K'_{1, \delta}(\mathfrak{n}) = \delta K_1(\mathfrak{n})\delta^{-1} \cap K_1(\mathfrak{n}) $$ of $ K_1(\mathfrak{n}).$
The subsets  $\Delta^\mathfrak{q}_1(\mathfrak{ n})$  act on any left $ \overline{\mathbb{F}}_p[{\rm GL}_2(\mathbb{F}_{p^2})]$-module via reduction modulo $\mathfrak{p}.$
There is  the following fact that is worth mentioning.
\begin{lemma}\label{lem2-2} 
Let  $\delta \in \Delta^\mathfrak{q}_1( \mathfrak{n})$. Then  there is a bijection between the coset space $ K_1(\mathfrak{n})/ K'_{1, \delta}(\mathfrak{n})$ and the orbit space  $  K_1(\mathfrak{n})\delta K_1(\mathfrak{n})/K_1(\mathfrak{n}) $ given as
\begin{eqnarray*}
K_1(\mathfrak{n})/ K'_{1, \delta}(\mathfrak{n}) &\rightarrow & K_1(\mathfrak{n})\delta K_1(\mathfrak{n})/K_1(\mathfrak{n}) \\
\lambda K'_{1, \delta}(\mathfrak{n})  &\mapsto&  \lambda\delta K_1(\mathfrak{n}).
\end{eqnarray*}
\end{lemma}
\begin{proof}
There is a surjective map $ K_1(\mathfrak{n}) \rightarrow  K_1(\mathfrak{n})\delta K_1(\mathfrak{n})/K_1(\mathfrak{n})$ which sends $ \lambda K'_{1, \delta}(\mathfrak{n})  $ to $ \lambda\delta K_1(\mathfrak{n}) .$ Two distinct elements $\lambda$ and  $ \lambda'$ map to the same orbit if and only of they lie in the same class modulo  $ K'_{1, \delta}(\mathfrak{n}).$
\end{proof}
For $\delta \in  \Delta^\mathfrak{q}_1(\mathfrak{n}),$ there are finitely many $\gamma_j \in \Delta^\mathfrak{q}_1(\mathfrak{n})$ such that  the double coset $ K_1( \mathfrak{n})\delta K_1(\mathfrak{n})$ decomposes as
$$ K_1(\mathfrak{n})\delta K_1(\mathfrak{n}) = \amalg_{j} \gamma_j K_1(\mathfrak{n}).$$ 
Let $ g \in Mat_2(\hat{\mathcal{O}})$ be such that its components at a finite number of finite places $\mathfrak{q}$ away from  $\mathfrak{pn}$ are of the form $\left(\begin{smallmatrix} \pi_\mathfrak{q} & 0 \\ 0 & 1
\end{smallmatrix}\right)$ or $ \left(\begin{smallmatrix} \pi^2_\mathfrak{q} & 0 \\ 0 & 1
\end{smallmatrix}\right) $  where $\pi_\mathfrak{q}$   is a uniformizer of $ \mathcal{O}_\mathfrak{q}$ and are the identity otherwise. When we denote $ \mathfrak{c} = (det(g))$ the ideal corresponding to $ g,$ then $ g \in  \Delta^\mathfrak{c}_1(\mathfrak{n}).$
\begin{lemma}\label{lem2_3_0}
Let $g \in \Delta^\mathfrak{c}_1(\mathfrak{n})$ as above.  Let $g_i$ corresponding to $[\mathfrak{b}_i]$ and $ K_1(\mathfrak{n})$ as above. Then,  for each  $ i$ there exist a unique index $j_i$,   $    1\leq j_i \leq  h, $    matrices  $ k_i = \left(\begin{smallmatrix} u_i  & 0 \\
0 & 1 \end{smallmatrix}\right) \in g_i K_1(\mathfrak{n}) g^{-1}_i $ and $ \beta_i := g_{j_i}g g^{-1}_i k_i = \left(\begin{smallmatrix} y_i & 0 \\
0 & 1\end{smallmatrix}\right)\in G(F)$ such that $ K_1(\mathfrak{n}) gK_1(\mathfrak{n}) = K_1(\mathfrak{ n}) g^{ -1}_{j_i}\beta_i g_{i}K_1(\mathfrak{n}).$
\end{lemma}
\begin{proof}
For each $i$ let $j_i$ be the unique index such that the ideal  $(det(g_{j_i} g  g^{-1}_{i }))$ is principal. Set then $ \alpha_i :=  g_{j_i} g  g^{-1}_{i}  = \left(\begin{smallmatrix} det(\alpha_i) &  0 \\
0 & 1\end{smallmatrix}\right).$ The ideal  $  (det(\alpha_i)) $ being  principal means that  $det(\alpha_i ) = x_i y_i$ with $ y_i \in F^* $ and $ x_i \in \hat{\mathcal{O}}^*.$ Set $  u_i = x^{-1}_i$ and define $ k_i = \left(\begin{smallmatrix} u_i & 0 \\ 0 & 1 \end{smallmatrix}\right)\in K_1( \mathfrak{n}).$ Then $ k_i \in g_i K_1(\mathfrak{n}) g^{-1}_i$ and   $\beta_i :=  \alpha_i k_i = \left(\begin{smallmatrix}  y_i & 0 \\ 0 & 1\end{smallmatrix}\right) \in G(F).$
Hence for each $i$ there exists a matrix  $$\beta_i \in  g_{j_i}\Delta^\mathfrak{c}_1(\mathfrak{n}) K_1(\mathfrak{n})g^{-1}_{i}\cap G(F) =g_{j_i}\Delta^\mathfrak{c}_1(\mathfrak{n}) g^{-1}_{i}\cap G(F) $$ such that $   K_1(\mathfrak{n}) g K_1(\mathfrak{n}) =  K_1(\mathfrak{n}) g^{-1}_{j_i} \beta_ig_{i}K_1(\mathfrak{n}).$ Indeed, $  g^{-1}_{j_i} \beta_i  g_{i} =  g^{-1}_{j_i} \alpha_i k_i  g_{i} =  g g^{-1}_{ i} k_i g_{i},$ and we observe that we have $ g^{-1}_{i} k_i g_{i} \in K_1( \mathfrak{n}).$ 
\end{proof}
For  $1\leq i \leq h,$ let  $ j_i$ and $ \beta_i$ as given in  the  above lemma. Let $\mathfrak{f}_i :=  (det(\beta_i)) = \mathfrak{b}_{j_i}\mathfrak{b}^{-1}_{i}\mathfrak{c}.$  Define $ \Lambda^{\mathfrak{f}_i}_{1, [\mathfrak{b}_i]}( \mathfrak{n})\index{$\Lambda^{\mathfrak{f}_i}_{1, [\mathfrak{b}_i]}( \mathfrak{n})$} := g_{j_i}\Delta^\mathfrak{c}_1(\mathfrak{n})g^{-1}_{i}\cap G(F).$ Explicitly this is the set
$$ \left\{  \left(\begin{smallmatrix} a   &  b \\ c  & d  \end{smallmatrix}\right) \in G( F): a  \in   \mathfrak{ b}_{j_i}\mathfrak{b}^{-1}_{i}, b  \in \mathfrak{b}_{j_i}, c \in \mathfrak{b}^{-1}_{ i}, d - 1  \in \mathfrak{n}\mathcal{O}; ( ad -bc)\mathcal{O} = \mathfrak{f}_i\right\}.$$
We set $ j := j_i.$ Let $ \alpha \in \Lambda^{\mathfrak{f}_i}_{1, [\mathfrak{b}_i]}( \mathfrak{n})$ ( we have in mind $\beta_i$). 
We consider the following  double coset $  \Gamma_{1, [ \mathfrak{b}_j]}(\mathfrak{n})\alpha \Gamma_{1, [\mathfrak{b}_i]}(\mathfrak{n}).$ This double coset defines a Hecke operator $ T_\alpha$ mapping  $$ {\rm  H}^r(\Gamma_{1, [ \mathfrak{b}_i]}(\mathfrak{n}), V^{a, b}_{r, s}( \overline{\mathbb{F}}_p)) \,\,\text{to}\,\,
 {\rm  H}^r(\Gamma_{1, [\mathfrak{b}_j]}(\mathfrak{n}), V^{a, b}_{r, s}( \overline{\mathbb{F}}_p))$$ as follows.
Firstly  one needs to introduce the following subgroups 
\begin{enumerate} 
\item $  \Gamma'^{,\alpha^{-1}}_{1, [\mathfrak{b}_i]}(\mathfrak{n})\index{$\Gamma'^{,\alpha^{-1}}_{1, [\mathfrak{b}_i]}(\mathfrak{n})$} := \Gamma_{1, [\mathfrak{b}_i]}( \mathfrak{ n})\cap \alpha^{-1}\Gamma_{1, [ \mathfrak{b}_j]}( \mathfrak{ n})\alpha $ 
\item $   \Gamma''^{, \alpha}_{1, [\mathfrak{b}_j]}(\mathfrak{ n})\index{$\Gamma'^{, \alpha}_{1, [\mathfrak{b}_j]}(\mathfrak{ n})$}  :=  \alpha\Gamma'_{1, [ \mathfrak{b}_i]}( \mathfrak{ n})\alpha^{-1} =  \alpha\Gamma_{1, [ \mathfrak{ b}_i]}( \mathfrak{ n})\alpha^{-1} \cap \Gamma_{1, [ \mathfrak{b}_j]}(\mathfrak{n}).$
\end{enumerate}
The operator $ T_\alpha$ is defined as the composition of the following maps:
$$ \begin{CD}  {\rm  H}^r(\Gamma'^{, \alpha^{-1}}_{1, [ \mathfrak{ b}_i]}( \mathfrak{ n}), V^{a, b}_{r, s}( \overline{\mathbb{F}}_p)) @>conj_\alpha>> {\rm  H}^r(\Gamma''^{, \alpha}_{1, [ \mathfrak{b}_j]}( \mathfrak{ n}), V^{a, b}_{r, s}( \overline{\mathbb{F}}_p)) \\
@AAresA   @VVcorV  \\
{\rm  H}^r(\Gamma_{1, [ \mathfrak{b}_i]}( \mathfrak{n}), V^{a, b}_{r, s}( \overline{\mathbb{F}}_p)) @. {\rm  H}^r(\Gamma_{1, [ \mathfrak{b}_j]}( \mathfrak{ n}), V^{a, b}_{r, s}( \overline{\mathbb{F}}_p)).                          \end{CD}
$$
Here $ res$ is the restriction map,  $ conj_\alpha$ is the isomorphism induced by the compatible maps: \begin{align*}
\Gamma''^{, \alpha}_{1, [ \mathfrak{b}_j]}( \mathfrak{n}) &\cong  \Gamma'^{, \alpha^{-1}}_{1, [ \mathfrak{ b}_i]}(\mathfrak{n})  \\ \omega  &\mapsto  \alpha^{-1}\omega\alpha
\end{align*}
and
\begin{align*}
V^{a, b}_{r, s}(\overline{\mathbb{F}}_p) &\rightarrow V^{a, b}_{r, s}( \overline{\mathbb{F}}_p)  \\ v &\mapsto \alpha.v.
\end{align*} 
Here $ cor$ is the corestriction homomorphism.
We explicitly describe $ T_\alpha$ in degree zero and one. In degree zero $T_\alpha$ is given as 
$$ \begin{CD}  {\rm H}^0( \Gamma'^{, \alpha^{-1}}_{1, [ \mathfrak{b}_i] } ( \mathfrak{ n}), V^{a, b}_{r, s}(\overline{\mathbb{F}}_p)) @>v \mapsto \alpha v>> {\rm  H}^0(\Gamma''^{, \alpha}_{1, [ \mathfrak{b}_j]}( \mathfrak{ n}), V^{a, b}_{r, s}( \overline{\mathbb{F}}_p)) \\
@AAv\mapsto vA   @VVv \mapsto \sum_{h  } h vV  \\
{\rm  H}^0(\Gamma_{1, [ \mathfrak{b}_i]}( \mathfrak{n}), V^{a, b}_{r, s}( \overline{\mathbb{F}}_p)) @. {\rm  H}^0(\Gamma_{1, [ \mathfrak{b}_j]}( \mathfrak{ n}), V^{a, b}_{r, s}( \overline{\mathbb{F}}_p)).                       \end{CD}
$$
where the  sum is over a set of cosets representatives of $\Gamma_{1,[\mathfrak{b}_j]}(\mathfrak{n})/\Gamma''^{, \alpha}_{ 1, [\mathfrak{b}_j]}( \mathfrak{n}).$ Hence one obtains that:
\begin{eqnarray*}
T_\alpha:{\rm  H}^0( \Gamma_{1, [\mathfrak{b}_i]}(\mathfrak{n}), V^{a, b}_{r, s}( \overline{\mathbb{F}}_p))    & \rightarrow & {\rm H}^0( \Gamma_{1, [\mathfrak{b}_j]}(\mathfrak{ n}), V^{a, b}_{r, s}( \overline{\mathbb{F}}_p)) \\
v  &\mapsto &  \sum_{\lambda \in \Gamma_{1,[\mathfrak{b}_j]}( \mathfrak{ n})/\Gamma''^{, \,\alpha}_{ 1, [\mathfrak{b}_j]}( \mathfrak{ n}) }(\lambda\alpha).v.
\end{eqnarray*}
It is worthwhile observing that the decomposition $\Gamma_{1,[\mathfrak{b}_j]}(\mathfrak{ n}) = \amalg_{r} \lambda_r \Gamma''^{, \alpha}_{1,[\mathfrak{b}_j]}(\mathfrak{n})$ is equivalent to the decomposition of the double cosets $ \Gamma_{1,[\mathfrak{b}_j]}(\mathfrak{ n})  \alpha \Gamma_{1,[\mathfrak{b}_i]}( \mathfrak{ n}) = \amalg_{r} \lambda_r \alpha\Gamma_{1,[\mathfrak{b}_i]}( \mathfrak{ n}).$
\subsubsection{Formula on degree one}
We now give the formula of $T_\alpha$ on degree one cohomology. To this end,  we need to recall the formulas describing the isomorphism $ conj_\alpha$ and the corestriction  in terms of non-homogeneous cocycles.
The  conjugation isomorphism is described by the formula \begin{eqnarray*} conj_\alpha: {\rm H}^1(\Gamma'^{, \alpha^{-1}}_{1, [ \mathfrak{b}_i]}(\mathfrak{ n}), V^{a, b}_{r, s}( \overline{\mathbb{F}}_p)) &\rightarrow & {\rm H}^1(\Gamma''^{, \alpha}_{1, [\mathfrak{b}_j]}(\mathfrak{n}), V^{a, b}_{r, s}( \overline{\mathbb{F}}_p))\\ c &\mapsto & (\omega \mapsto \alpha.c(\alpha^{-1} \omega \alpha)).
\end{eqnarray*}
For the corestriction homomorphism, let $\Gamma_{1, [\mathfrak{b}_j]}(\mathfrak{n}) = \amalg_n \gamma_n\Gamma''^{,  \alpha}_{1, [ \mathfrak{b}_j]}( \mathfrak{n}).$ For $ \omega \in  \Gamma_{1, [\mathfrak{b}_j]}( \mathfrak{n}),$ let $  s_n  $ be the unique index such that  $ \gamma^{-1}_n \omega \gamma_{s_n} \in \Gamma_{1, [\mathfrak{b}_j]}(\mathfrak{ n}).$
Then the corestriction homomorphism is given as 
\begin{eqnarray*} cor: {\rm H}^1(\Gamma''^{, \alpha}_{1, [ \mathfrak{b}_j]}(\mathfrak{ n}), V^{a, b}_{r, s}( \overline{\mathbb{F}}_p))  & \rightarrow & {\rm H}^1(\Gamma_{1, [\mathfrak{b}_j]}( \mathfrak{n}), V^{a, b}_{r, s}( \overline{\mathbb{F}}_p)) \\ c & \mapsto &  (\omega \mapsto \sum_n \gamma_n.c( \gamma^{-1}_n \omega \gamma_{s_n})).
\end{eqnarray*}\label{Indep}
The formula of $ T_\alpha$ on degree one cohomology is thus
\begin{eqnarray*}
T_\alpha:{\rm  H}^1( \Gamma_{1, [\mathfrak{b}_i]}(\mathfrak{n}), V^{a, b}_{r, s}( \overline{\mathbb{F}}_p))    & \rightarrow & {\rm H}^1( \Gamma_{1, [\mathfrak{b}_j]}(\mathfrak{ n}), V^{a, b}_{r, s}( \overline{\mathbb{F}}_p)) \\
c  &\mapsto &  (\omega \mapsto \sum_n\gamma_n\alpha.c((\gamma_n\alpha)^{-1} \omega \gamma_{ s_n}\alpha)).
\end{eqnarray*}
Indeed with the given formulas we have
\begin{eqnarray*}
(cor( conj_\alpha(c) ))( w)   &= & \sum_{\gamma_n \in \Gamma_{1, [\mathfrak{b}_j]}(\mathfrak{n})/ \Gamma'_{1, [\mathfrak{b}_j] } (\mathfrak{n})} \gamma_n. ( conj_\alpha(c)( \gamma^{-1}_n w \gamma_{s_n} ))  \\
& = &  \sum_{\gamma_n \in \Gamma_{1, [\mathfrak{b}_j]}(\mathfrak{n})/\Gamma'_{1, [\mathfrak{b}_j]} (\mathfrak{n})} \gamma_n \alpha. c( \alpha^{-1}\gamma^{-1}_n w \gamma_{s_n} \alpha).
\end{eqnarray*}
Let  $ \lambda_i$  be another set of representatives of  $ \Gamma_{1, [\mathfrak{b}_j]}(\mathfrak{n})/\Gamma''^{, \alpha}_{1, [\mathfrak{ b}_j]}(\mathfrak{n}),$ and $ \sigma_i \in \Gamma''^{, \alpha}_{1, [\mathfrak{b}_j]}$ such that $  \lambda_i = \gamma_i \sigma_i.$ With this we have
$$ (cor( c))( w) = \gamma_i \sigma. c( \sigma^{-1}_i \gamma^{-1}_i w \gamma_{j_i} \sigma_i).$$
Because taking conjugation by an element from $ \Gamma'^{, \alpha}_{1, [\mathfrak{b}_j]}(\mathfrak{n})$ gives cohomologous cocycle, we deduce that  the corestriction map does not depend on the choice of representatives of $ \Gamma_{1, [\mathfrak{b}_j]}(\mathfrak{n})/\Gamma''^{, \alpha}_{1, [\mathfrak{b}_j]}(\mathfrak{n}).$ This means that $T_\alpha$ does not depend on the choice of set of representatives and so only depends on the double coset $ \Gamma_{1, [\mathfrak{b}_j]} \alpha \Gamma_{1, [\mathfrak{b}_i]}  $ since we  know that
$$ \Gamma_{1, [\mathfrak{b}_j]}(\mathfrak{n}) = \amalg_n \gamma_n \Gamma''^{, \alpha}_{1, [\mathfrak{b}_j]}      \Longleftrightarrow  \Gamma_{1, [\mathfrak{b}_j]}(\mathfrak{n}) \alpha \Gamma_{1, [\mathfrak{b}_i]}(\mathfrak{n}) = \amalg_n \gamma_n \alpha \Gamma_{1, [\mathfrak{b}_i]}(\mathfrak{n}).$$
\subsubsection{Action of  $ T_g $ on  $  \oplus^h_{i = 1} {\rm H }^r (\Gamma_{1, [\mathfrak{b}_i]}( \mathfrak{ n}),  V^{a, b}_{r, s}( \overline{\mathbb{F}}_p))$}
Now that we have recalled the formulas of the Hecke operators on  group cohomology, let us say how Hecke operators act on the $\overline{\mathbb{F}}_p$-vector spaces $ \oplus^h_{i = 1} {\rm H }^r (\Gamma_{1, [\mathfrak{b}_i]}( \mathfrak{n}),  V^{a, b}_{r, s}(\overline{\mathbb{F}}_p)).$ Let  $g$ be as in  Lemma \ref{lem2_3_0} and  consider  $ \beta_i$ and $ j_i$ provided by the lemma loc. cit.  Let $T_{\beta_i}$ the Hecke operator corresponding to  the double coset $\Gamma_{1, [\mathfrak{b}_{j_i}]}(\mathfrak{n})\beta_i \Gamma_{1, [\mathfrak{b}_{i}]}( \mathfrak{n}).$  Then  $T_{\beta_i}$ sends an element from  ${\rm  H}^r(\Gamma_{1, [\mathfrak{b}_i]}( \mathfrak{n}), V^{a, b}_{r, s}( \overline{\mathbb{F}}_p))$ to   $ {\rm H}^r(\Gamma_{1, [\mathfrak{b}_{j_i}]}( \mathfrak{n}), V^{a, b}_{r, s}( \overline{\mathbb{F}}_p)).$ It was proved by Shimura, see \cite{Shimura},  that for  $(c_1, \cdots, c_h) \in \oplus^h_{i = 1}{\rm H }^r (\Gamma_{1, [\mathfrak{b}_i]}(\mathfrak{n}),  V^{a, b}_{r, s}( \overline{\mathbb{F}}_p)),$  the Hecke action of $ T_g$ is 
$$T_g.( c_1, \cdots, c_h) = ( d_1, \cdots.d_h),$$
where $ d_{ j_i}  =   T_{\beta_i}.c_i.$
\begin{remark}
In the idyllic situation where  the ideal  $(det( g))$ is principal, then, the  Hecke operator $ T_g$ does not permute the summands in  $   \oplus^h_{l = 1} {\rm  H}^r (\Gamma_{1, [\mathfrak{b}_l]}(\mathfrak{ n}),  V^{a, b}_{r, s}( \overline{\mathbb{F}}_p) ).$
Indeed $ ( det( g_{j_i} g g^{-1}_{i} )) = (det (g)),$ so $ j_i =  i $ in Lemma \ref{lem2_3_0}. Therefore  $ T_g.(c_1, \cdots, c_h) = ( d_1, \cdots, d_h)  $ where  $ d_i = T_{\beta_i}.c_i.$
\end{remark}
\begin{remark}
Let $ g$ be as in Lemma \ref{lem2_3_0}. Let  us denote the ideal $ (det(g))$ as $ \mathfrak{c}.$ 
Then  $T_g$ maps the $\overline{\mathbb{F}}_p$-vector spaces  $ \oplus^r_{l = 1} {\rm H}^r(\Gamma_{1, [\mathfrak{b}_l]}( \mathfrak{n}), V^{a, b}_{r, s}( \overline{\mathbb{F}}_p))$ to $  \oplus^r_{l = 1} {\rm H}^r(\Gamma_{1, [\mathfrak{c}^{-1}\mathfrak{b}_l]}( \mathfrak{ n}), V^{a, b}_{r, s}( \overline{\mathbb{F}}_p)).$ To see this,  one needs to just recall that $ T_g$ maps $$ \oplus^r_{i = 1} {\rm  H}^r(\Gamma_{1, [\mathfrak{b}_i]}(\mathfrak{n}), V^{a, b}_{r, s}( \overline{\mathbb{F}}_p)) \,\, \text{to}\,\,    \oplus^r_{ i = 1} {\rm  H}^r( \Gamma_{ 1, [\mathfrak{b}_{j_i}]}(\mathfrak{n}), V^{a, b}_{r, s}( \overline{\mathbb{F}}_p))$$ where $  j_i$ is such that $ (det( g_{j_i}g g^{-1}_{i}))$ is principal. In terms of ideals this means that  $ [\mathfrak{c}^{-1}\mathfrak{b}_i] = [\mathfrak{b}_{j_i}].$
\end{remark}\label{DiamondOp}
We shall next recall a definition of a class of degree one Hecke operators known as {\it diamond operators}\index{diamond operators}.
\subsubsection{Diamond operators}
This is a special kind (degree one Hecke operator) of Hecke operator defined as follows. Define  the open compact subgroup $ K_0( \mathfrak{n})\index{$K_0(\mathfrak{n})$}$ of $ G(\hat{\mathcal{O}})  $ as
$$K_0(\mathfrak{n}) = \big\{  \left(\begin{smallmatrix} a & b \\ c & d \end{smallmatrix}\right) \in \prod_{ \mathfrak{ q}\nmid \infty} G( \mathcal{O}_\mathfrak{q}): c   \in \mathfrak{n}\hat{\mathcal{ O}}\big\}. $$
Then  $ K_1(\mathfrak{n}) $ is a  normal subgroup of $ K_0(\mathfrak{n}).$ So for any $\alpha \in K_0(\mathfrak{n})$ we have
$$ \alpha K_1(\mathfrak{n})\alpha^{-1}   = K_1(\mathfrak{n}).$$
Therefore we deduce that $ K_1(\mathfrak{n})\alpha K_1(\mathfrak{n}) = \alpha K_1(\mathfrak{n}).$
The Hecke operator corresponding to the double coset $K_1(\mathfrak{n})\alpha K_1(\mathfrak{n})$ is called a {\it diamond  operator}.
\begin{example}
Take $\alpha \in K_0(\mathfrak{n})$ with determinant corresponding to a principal ideal  such  that  at one place $\mathfrak{q}$ dividing $\mathfrak{n}$  the component $ \alpha_\mathfrak{q} $ has reduction modulo $\mathfrak{n}$ a matrix of the form $ \left(\begin{smallmatrix} \omega & 0 \\ 0 & \gamma \end{smallmatrix}\right)
$ and at the other places the components are the identity matrix.
Because the determinant of $\alpha$ is principal, the Hecke operator $T_\alpha$ does not permute the components:
\begin{eqnarray*}
T_\alpha: \oplus^h_{i = 1}{\rm H}^1(\Gamma_{1, [\mathfrak{b}_i]}(\mathfrak{n}), V^{a, b}_{r, s}( \overline{\mathbb{F}}_p)) & \rightarrow& \oplus^h_{i = 1}{\rm H}^1(\Gamma_{1, [\mathfrak{b}_i]}(\mathfrak{n}), V^{a, b}_{r, s}( \overline{\mathbb{F}}_p)) \\
(c_1, \cdots, c_h) &\mapsto& (T_{\beta_1}.c_1, \cdots, T_{\beta_h}.c_h)
\end{eqnarray*}
where $ \beta_i \in \Gamma_{0, [\mathfrak{b}_i]}(\mathfrak{n}) := g_iK_0(\mathfrak{n})g^{-1}_i\cap  G(F) $ such that $ \alpha K_1(\mathfrak{n}) = g_i^{-1}\beta_i g_iK_1(\mathfrak{n})$ and $T_{\beta_i}$ is the Hecke operator corresponding to the double coset $\Gamma_{1, [\mathfrak{b}_i]}(\mathfrak{n})\beta_i\Gamma_{1, [\mathfrak{b}_i]}(\mathfrak{n}).$ Note that $T_{\beta_i}$ defines a non-adelic diamond operator. More explicitly
the Hecke operator  $T_{\beta_i}$ on  $ {\rm H}^1(\Gamma_{1, [\mathfrak{b}_i]}(\mathfrak{n}),  V^{a, b}_{r, s}( \overline{\mathbb{F}}_p))  $ is given as
\begin{eqnarray*}
T_{\beta_i} : {\rm H}^1(\Gamma_{1, [\mathfrak{b}_i]}(\mathfrak{n}),  V^{a, b}_{r, s}( \overline{\mathbb{F}}_p)) &\rightarrow& {\rm H}^1(\Gamma_{1, [\mathfrak{b}_i]}(\mathfrak{n}),  V^{a, b}_{r, s}( \overline{\mathbb{F}}_p)) \\
 c &\mapsto&  (w  \mapsto  \beta_i. c(\beta^{-1}_i w \beta_i)).
\end{eqnarray*}
\end{example}
Aside from this interesting fact, there is a nice interpretation of diamond operators as  in the classical setting. It arises from the isomorphism of  abelian groups
\begin{eqnarray*} K_0(\mathfrak{n})/K_1(\mathfrak{n})  & \rightarrow  &  (\hat{\mathcal{O}}/\mathfrak{n}\hat{\mathcal{O}})^* \cong  (\mathcal{O}/\mathfrak{n})^* \\
\left(\begin{smallmatrix} a  & b \\ c & d
\end{smallmatrix}\right)  & \mapsto &  d \pmod{\mathfrak{n}}.
\end{eqnarray*}
This means that we have an action of the group $(\mathcal{O}/\mathfrak{n})^*$ on  $ \oplus^h_{i = 1}{\rm H}^1( \Gamma_{1, [\mathfrak{b}_i]}(\mathfrak{n}), V^{a, b}_{r, s}( \overline{\mathbb{F}}_p)).$ Let  $ \chi: (\mathcal{O}/\mathfrak{n})^* \rightarrow \overline{\mathbb{F}}^*_p$ be a character. As  a representation of the abelian group  $ (\mathcal{O}/\mathfrak{n})^*, $ then when $ p \nmid \sharp(\mathcal{O}/\mathfrak{n})^*,$ the space $ \oplus^h_{   i = 1}{\rm H}^1( \Gamma_{1, [\mathfrak{b}_i]}(\mathfrak{n}), V^{a, b}_{r, s}( \overline{\mathbb{F}}_p)$  decomposes as a direct sum  of $  \chi$-eigenspaces. So  by  denoting the spaces $ \oplus^h_{i = 1}{\rm  H}^1(\Gamma_{1, [\mathfrak{b}_i]}(\mathfrak{ n}), V^{a, b}_{r, s}( \overline{\mathbb{F}}_p)$  as  $ \mathcal{M}_{V^{a, b}_{r, s}( \overline{\mathbb{F}}_p)}(\mathfrak{n})$  and a $\chi$-eigenspace as $ \mathcal{M}_{V^{a, b}_{r, s}( \overline{\mathbb{F}}_p)}(\mathfrak{n}, \chi),$  then we have $$ \mathcal{M}_{V^{a, b}_{r, s}( \overline{\mathbb{F}}_p)}( \mathfrak{n}) = \oplus_{\chi}\mathcal{M}_{V^{a, b}_{r, s}( \overline{\mathbb{F}}_p)}(\mathfrak{n}, \chi).$$
Let us turn next to the definition  of the Hecke algebra.
\subsubsection{Hecke algebra}
We start by defining first what we call mod $p$ cohomological modular forms over $ F.$ Recall that we have denoted  $ p\mathcal{ O}$ as $ \mathfrak{p}$ and we are assuming that $p$ is inert in $ F.$ The residue field is then $ \mathbb{F}_{p^2}.$ 
The congruence subgroups $  \Gamma_{1, [\mathfrak{b}_i]}(\mathfrak{ n}) $ act on $ V^{ a, b}_{r, s}(  \overline{ \mathbb{F}}_p)$ via reduction modulo $p.$
\begin{definition}\label{def2-0}
A {\it cohomological mod $ p$ modular form}\index{cohomological mod $ p$ modular form} of weight  $ V^{a, b}_{r, s}( \overline{\mathbb{F}}_p)$ and level $\mathfrak{n}$ over $ F$ is a class in \[\oplus^h_{i = 1}{\rm H}^1(\Gamma_{1, [\mathfrak{b}_i]}(\mathfrak{ n}), V^{a, b}_{r, s}( \overline{\mathbb{F}}_p)).\]
We have denoted this $\overline{\mathbb{F}}_p$-vector spaces  as $\mathcal{M}_{V^{a, b}_{r, s}( \overline{\mathbb{F}}_p)}( \mathfrak{n})\index{$\mathcal{M}_{V^{a, b}_{r, s}( \overline{\mathbb{F}}_p)}( \mathfrak{n})$,\,\text{mod $ p$ modular forms}}.$
\end{definition}
We next define the  Hecke algebra of interest for our purposes.
\begin{definition}[Hecke algebra]
\begin{enumerate}
\item The abstract Hecke algebra  $\mathcal{H}\index{$\mathcal{H}, \,\text{abstract Hecke algebra}$}$ is the polynomial algebra $\mathbb{Z}[T_\mathfrak{q},  S_\mathfrak{ q}|\,\,\mathfrak{q} \nmid \mathfrak{pn} \, \textrm{maximal ideal} \, \subset \mathcal{O}].$
\item The Hecke algebra  $ \mathcal{H}(\mathcal{M}_{V^{a, b}_{r, s}( \overline{\mathbb{F}}_p)}( \mathfrak{n}))\index{$\mathcal{H}(\mathcal{M}_{V^{a, b}_{r, s}( \overline{\mathbb{F}}_p)}( \mathfrak{n})),\,\text{Hecke algebra}$}$ acting on  $ \mathcal{M}_{V^{a, b}_{r, s}( \overline{\mathbb{F}}_p)}( \mathfrak{n})$ is the homomorphic image of: $\mathcal{H} \rightarrow  End_{\overline{\mathbb{F}}_p} (\mathcal{M}_{V^{a, b}_{r, s}( \overline{\mathbb{F}}_p)}( \mathfrak{n})));   T_\mathfrak{ q},  S_\mathfrak{ q} \mapsto  T_\mathfrak{q}, S_\mathfrak{ q}.$
\end{enumerate}
\end{definition}
As we said an eigenform for all the Hecke operator $ T_\mathfrak{ q}$ for $ \mathfrak{ q} $ away from $ \mathfrak{ pn}$ gives rise to a system of Hecke eigenvalues. Here is a formal definition of a system of Hecke eigenvalues with values in  $ \overline{\mathbb{F}}_p.$
\begin{definition}
A system of Hecke eigenvalues\index{Hecke! eigenvalues} with values in $\overline{\mathbb{F}}_p$ is a ring homomorphism
$ \psi: \mathcal{H} \rightarrow \overline{\mathbb{ F}}_p.$
We say that it occurs in  $\mathcal{M}_{V^{a, b}_{r, s}( \overline{\mathbb{F}}_p)}( \mathfrak{n}) $ if there is a non-zero  $ f \in  \mathcal{M}_{V^{a, b}_{r, s}( \overline{\mathbb{F}}_p)}( \mathfrak{n})$ such that $ Tf = \psi( T)f$ for all $ T\in \mathcal{H}.$
\end{definition}
In the next section we shall relate the induced modules $ {\rm Ind}^{\Gamma_{1,[\mathfrak{b}_i]}(\mathfrak{n})}_{\Gamma_{1, [\mathfrak{b}_i]}(\mathfrak{pn})}( \overline{\mathbb{F}}_p) $ to some irreducible $\overline{\mathbb{F}}_p[{\rm GL}_2( \mathcal{O})]$-modules of the form $ V^{a, b}_{r, s}(\overline{\mathbb{F}}_p).$
\section{The relevant induced modules}\label{sec2_3}
We recall that by assumption  we have fixed a rational inert prime $ p$  and $\mathfrak{p} = p\mathcal{O}$ does not divide an integral ideal $\mathfrak{n}$ which was also fixed.
Here we will be concerned with the  induced modules\index{induced modules}  $  {\rm Ind}^{\Gamma_{1,[\mathfrak{b}_i]}(\mathfrak{n})}_{\Gamma_{1, [\mathfrak{b}_i]}(\mathfrak{pn})}( \overline{\mathbb{F}}_p)\index{${\rm Ind}^{\Gamma_{1,[\mathfrak{b}_i]}(\mathfrak{n})}_{\Gamma_{1, [\mathfrak{b}_i]}(\mathfrak{pn})}( \overline{\mathbb{F}}_p)$,\,\text{induced module}}.$ We shall derive a more explicit decomposition of the latter. Let $ \tilde{G} = {\rm  GL}_2( \mathbb{F}_{p^2}) $ and $ \tilde{S} =  {\rm  SL}_2( \mathbb{F}_{p^2}).$

Define the following congruence subgroups of $ {\rm GL}_2(F):$
\[ \Gamma^1_{1, [\mathfrak{b}_i]}( \mathfrak{n}  )\index{$\Gamma^1_{1, [\mathfrak{b}_i]}( \mathfrak{n}$)} := g_i K_1( \mathfrak{n}) g^{-1}_i\cap {\rm SL}_2(F).\]
Because $ \left(\begin{smallmatrix} t_i & 0 \\ 0 & 1 \end{smallmatrix}\right)\left(\begin{smallmatrix} a & b \\ c & d \end{smallmatrix}\right)\left(\begin{smallmatrix} t^{-1}_i & 0 \\ 0 & 1 \end{smallmatrix}\right) = \left(\begin{smallmatrix} a & t_ib \\ t^{-1}_ic & d \end{smallmatrix}\right), $ one obtains that
$$ \Gamma^1_{1, [\mathfrak{b}_i]}(\mathfrak{n}) = \left\{\left(\begin{smallmatrix} a & b \\ c & d \end{smallmatrix}\right)\in {\rm SL}_2(F)  : a -1, d-1 \in \mathfrak{n}; b \in \mathfrak{b}_i,  c \in \mathfrak{b}^{-1}_i\mathfrak{n}  \right\}.$$
In particular with our assumptions one has that
$$ \Gamma^1_{1, [\mathfrak{b}_1]}(\mathfrak{n}) = \Gamma_1( \mathfrak{ n}) := \left\{\left(\begin{smallmatrix} a  & b \\ c  & d \end{smallmatrix}\right) \in {\rm SL}_2(\mathcal{O}): a-1,\, d-1,\,   c \equiv 0 \pmod{\mathfrak{n}}  \right\}.$$
Furthermore, let
\[\Gamma( \mathfrak{n}) = \left\{\left(\begin{smallmatrix} a  & b \\ c  & d \end{smallmatrix}\right) \in {\rm SL }_2(\mathcal{O}): a-1,\, d-1,\,  b,\,  c \equiv 0 \pmod{\mathfrak{n}}  \right\}.\]
\begin{lemma}\label{lem1} Let $ \tilde{S} =  {\rm  SL}_2( \mathbb{F}_{p^2}).$ Then, we have  an exact sequence \[  1\rightarrow\Gamma( \mathfrak{ p})\cap \Gamma_1(\mathfrak{n}) \rightarrow \Gamma_1 ( \mathfrak{ n})\rightarrow \tilde{S} \rightarrow 1\] where the  third arrow  is reduction modulo $ p.$
\end{lemma}
\begin{proof} 
It is clear that $ \Gamma(\mathfrak{ p})\cap \Gamma_1(\mathfrak{n})$ is the kernel of the reduction modulo $ \mathfrak{ p}$ of $ \Gamma_1(\mathfrak{ n}).$ So, we are left to see the surjectivity of the third  arrow. To this end let  $a, b ,  c , d \in \mathcal{ O}$ with $ ad - bc \equiv 1 \pmod{ \mathfrak{ p}}.$ We need to find $ \alpha, \beta, \gamma, \delta \in \mathcal{ O}$ such that $ \alpha\delta - \beta\gamma = 1$ with the congruences: \begin{eqnarray*}
\alpha  & \equiv & a \pmod{ \mathfrak{ p}} \\                                                                          \alpha & \equiv &  1 \pmod{ \mathfrak{ n}} \\ 
\beta   & \equiv & b \pmod {  \mathfrak{ p}  } \\
\gamma & \equiv & c \pmod {  \mathfrak{ p}  } \\
\gamma  & \equiv & 0  \pmod{ \mathfrak{ n}} \\
\delta & \equiv & d  \pmod {  \mathfrak{ p}  } \\
\delta &\equiv  & 1 \pmod {  \mathfrak{ n} }.
\end{eqnarray*}
It is readily seen that if   $  0\neq c  \in  \mathfrak{ n}$ and is coprime with $ \mathfrak{ p}$  then the Chinese Remainder Theorem  permits to conclude. Indeed,  set $ \gamma = c,$  there exist $ \alpha, \delta \in \mathcal{ O}$ with $ \alpha \equiv a \pmod{  \mathfrak p},  $  $ \alpha \equiv  1 \pmod{  \gamma},$ $ \delta \equiv d \pmod{ \mathfrak{ p}}, $ $ \delta \equiv 1 \pmod{ \gamma}.$ This gives $  \alpha\delta \equiv  1 \pmod{ \gamma},$ and so there exists $ \beta \in \mathcal{O} $ such that $ \alpha\delta - \beta\gamma = 1$ and $ \beta \equiv b \pmod{ \mathfrak{ p}}.$ So we need to see that we can always reduce to this case. To this end as  $ \mathfrak{n}  $ is coprime with $ \mathfrak{ p},$ we can find $  n \in \mathfrak{ n}, \, k \in \mathfrak{ p }, $ and $  r,  s \in \mathcal{ O}$ such that $ nr -  ks  = 1.$ The image of the matrix $ \left( \begin{smallmatrix} r&0\\ n& n \end{smallmatrix}\right)$ belongs to $ \tilde{ S}$ and can be lifted by the previous arguments. Then $   \left( \begin{smallmatrix} r&0\\ n&n \end{smallmatrix}\right)\left( \begin{smallmatrix} a&b\\ c&d \end{smallmatrix}\right)  =  \left( \begin{smallmatrix}  ra& rb\\ n(c + a )& n(d + b) \end{smallmatrix}\right)$ is a matrix in   $ \tilde{ S}$ whose bottom line has entries in $ \mathfrak{ n}.$  Then if  $n(c + a )\neq 0 $  we are done, otherwise we just have to multiply from the right by   $\left( \begin{smallmatrix} 1&0\\ n&1 \end{smallmatrix}\right)$ for the condition to hold.
\end{proof}
\begin{corollary}\label{cor1}
For $  i \geq 1,$  the congruence subgroup $ \Gamma^1_{1, [\mathfrak{b}_i]}(\mathfrak{n})$ surjects onto $ \tilde{S}$ via reduction modulo $p.$
\end{corollary}
\begin{proof}
From Lemma \ref{lem1}, we have that $ \Gamma_1(\mathfrak{n})$ surjects onto $ \tilde{S}.$ So let  $ M = \left(\begin{smallmatrix} a & b \\ c & d \end{smallmatrix}\right) \in \tilde{S} 
$  and $  \left(\begin{smallmatrix} \alpha & \beta \\ \gamma & \delta \end{smallmatrix}\right) \in \Gamma_1(\mathfrak{n})$ be a lift of  $ M.$ For each  $ i > 1$ take $ \lambda_i \in \mathfrak{b}_i$ such that $\lambda_i \equiv  1 \pmod{\mathfrak{p}}$ ( this is possible since $ \mathfrak{b}_i$ is coprime with $ \mathfrak{p}$). Then the matrix $ \left(\begin{smallmatrix} \alpha & \lambda_i\beta \\ \lambda^{-1}_i\gamma & \delta \end{smallmatrix}\right)$ belongs to $ \Gamma^1_{1, [\mathfrak{b}_i]}(\mathfrak{n})$ and it has reduction $M.$
\end{proof}
From the fact that $  \Gamma^1_{1, [ \mathfrak{b}_i ]}(\mathfrak{n}) \subset \Gamma_{1,[ \mathfrak{b}_i] }(\mathfrak{n}),$ we deduce that the reduction modulo $ p$ of $  \Gamma_{1, [\mathfrak{b}_i]}(\mathfrak{n}) $ contains $ \tilde{S}.$ Now suppose we are given two subgroups $ H_1, H_2$ of $  \tilde{G} = {\rm GL}_2( \mathbb{F}_{p^2} )   $ containing $ \tilde{S}$ and such that their images by the determinant map are the same: $det(H_1) = det(H_2) < \mathbb{F}^*_{p^2}.$ The fact  $ det( H_1\cap H_2) = det(H_1)\cap det(H_2)$ implies the following commutative diagram with exact rows:
$$\begin{CD}
 1 @>>>\tilde{S} @>>> H_1\cap H_2 @> det>>  det(H_1) @>>> 1   \\
@.   @VVV @VVV @VVV \\
1 @>>>\tilde{S} @>>> H_1 @>det>>  det(H_1) @>>> 1. 
\end{CD}
$$  
Therefore one has $ H_1 = H_2, $ and we have established that any subgroup  $H$ of $  \tilde{G} $ containing $ \tilde{S}$ is uniquely determined by the image of the determinant map $  H \xrightarrow{det}  \mathbb{F}^*_{p^2}.$ From this fact we derive that  $  \Gamma_{1, [\mathfrak{b}_i]}( \mathfrak{n})$ reduces to  $$ 
\mathcal{T}_1(\mathfrak{n}):= \left\{ g \in \tilde{G}: det(g) \in Im( \mathcal{O}^*\xrightarrow{reduction} \mathbb{F}^*_{p^2}) \right\}.$$
We also derive that $ \Gamma_{1, [\mathfrak{b}_i]}( \mathfrak{pn})$ reduces to 
$$\mathcal{T}_1(\mathfrak{pn}):= \left\{ \left(\begin{smallmatrix} a &  b \\ 0  & 1                                            \end{smallmatrix}\right)\in \tilde{G}: a \in Im( \mathcal{O}^*\xrightarrow{reduction} \mathbb{F}^*_{p^2}) \right\}. $$
In summary, reduction mod  $ p$ gives the following bijection:
$$  \Gamma_{1, [\mathfrak{b}_i]}(\mathfrak{pn })\backslash \Gamma_{1, [\mathfrak{b}_i]}(\mathfrak{n})\rightarrow \mathcal{T}_1(\mathfrak{pn})\backslash \mathcal{T}_1(\mathfrak{n}).$$
Define  $  \tilde{U}\index{$\tilde{U}$} =\left\{ \left(\begin{smallmatrix} a  & b \\ 0  &  1 \end{smallmatrix}\right) \in \tilde{G} \right\}.$ Next we have the following bijection
\begin{eqnarray*}
\mathcal{T}_1(\mathfrak{pn})\backslash \mathcal{T}_1(\mathfrak{n})  & \longleftrightarrow &  \tilde{U}\backslash \tilde{G} \\
\mathcal{T}_1(\mathfrak{pn}) g  & \mapsto  & \tilde{U} g.
\end{eqnarray*}
Indeed the map is surjective and two elements from $\mathcal{T}_1(\mathfrak{ n })$ are sent to the same class modulo $ \tilde{U}$ if and only they belong to the same class modulo $\mathcal{T}_1(\mathfrak{ pn }) $ because we have $ \tilde{U}\cap \mathcal{T}_1(\mathfrak{ n }) = \mathcal{T}_1(\mathfrak{ pn }).$
Composing these two bijections, we obtain the bijection
$$ \Gamma_{1, [\mathfrak{b}_i]}(\mathfrak{pn })\backslash \Gamma_{1, [\mathfrak{b}_i]}(\mathfrak{n})\longleftrightarrow \tilde{U}\backslash \tilde{G}.$$
\subsection{Induced modules}
Let  $H$ be a group and $ J < H $ a subgroup of finite index. For a  left $ J$-module $ M$ the induced module,  and a twisted  induced module are defined as follows.
\begin{definition}
\begin{enumerate}
\item $ {\rm Ind}^H_J( M) = \{f : H \rightarrow M :  f( gh) = gf( h) \,\,\forall \,\, g\in J, h \in H  \}.$ 
\item  Given a character $  \chi : J \rightarrow  \overline{\mathbb{F}}^*_p,$ we define a twisted induced  module\index{twisted induced  module} as
$$ {\rm Ind }^H_J (\overline{\mathbb{F}}^\chi_p) = \{f : H \rightarrow \overline{\mathbb{F}}_p :  f( gh) = \chi(g)f( h) \,\,\forall \,\, g\in J, h \in H    \}.$$
\end{enumerate}
\end{definition}
Recall how a left action of $H$  on $ {\rm Ind}^H_J(M)$ can be defined: for    $ g \in  H $ and  $ f \in {\rm Ind}^H_J(M)$ we have  $(g.f)( h) := f(hg).$

Let $\tilde{B}\index{$\tilde{B},$\,\text{Borel subgroup}} =  \left\{\left(\begin{smallmatrix} a   & b \\ 0    &  e \end{smallmatrix} \right)\in \tilde{ G}\right\}   $ be the Borel subgroup of $ \tilde{G} $ and define the character $ \chi$ of  $ \tilde{B}$ by  
\begin{eqnarray*}  \chi: \tilde{B} &\rightarrow & \mathbb{F}^*_{p^2} \\
\left(\begin{smallmatrix} a   & b \\ 0    &  e \end{smallmatrix} \right)&\mapsto&    e.
\end{eqnarray*}
For an integer $ d,$ we also set $  \chi^d( . ) =   (\chi(.))^d.$ The  homomorphism  $\chi$ induces  a group isomorphism $$   \tilde{U}\backslash \tilde{B} \cong  \mathbb{F}^*_{p^2}.$$ From this isomorphism we obtain the following isomorphism of $ \tilde{B}$-modules $$  {\rm Ind}^{\tilde{B}}_{\tilde{U}}(\overline{\mathbb{F}}_p) \cong   {\rm Ind}^{\mathbb{F}^*_{p^2}}_{\{1\}}(\overline{\mathbb{F}}_p).$$ The isomorphism is defined as follows:
\begin{eqnarray*}
\Phi: {\rm Ind}^{\mathbb{F}^*_{p^2}}_{\{1\}}(\overline{\mathbb{F}}_p) & \rightarrow &   {\rm Ind}^{\tilde{B}}_{\tilde{U}}(\overline{\mathbb{F}}_p) \\
f    & \mapsto &\big( \left(\begin{smallmatrix} a & b \\ 0  & e 
\end{smallmatrix}\right) \mapsto f(e)\big).
\end{eqnarray*}
The representation $ {\rm Ind}^{\mathbb{F}^*_{p^2}}_{\{1\}}(\overline{\mathbb{F}}_p)$ is the regular representation of $ \mathbb{F}^*_{p^2}.$ This is  a  $ ( p^2 -1)$-dimensional representation of an abelian group of order prime to $ p$ and hence it admits a decomposition into a direct sum of one-dimensional representations of $ \mathbb{F}^*_{p^2}.$ By  a slight abuse of notation, the summands are the  $ \mathbb{F}^*_{p^2}$-modules $  \overline{\mathbb{F}}^{\chi^d}_p,$ where for $  x\in  \mathbb{ F}^*_{p^2},  y \in \overline{\mathbb{F}}_p,$ we have  $ x.y := x^dy$ with $ 0\leq d \leq p^2 -2.$ 
\begin{proposition}\label{Prop2-2}
For all  $i,$ there is the following isomorphism of left $\Gamma_{1, [\mathfrak{b}_i]}(\mathfrak{n})$-modules and left $\Gamma^1_{1, [\mathfrak{b}_i]}(\mathfrak{n})$-modules respectively:
\begin{enumerate}
\item ${\rm Ind}^{\Gamma_{1, [\mathfrak{b}_i]}(\mathfrak{n})}_{\Gamma_{ 1, [\mathfrak{b}_i]}(\mathfrak{ pn})}(\overline{\mathbb{F}}_p) \cong \oplus^{ p^2 -2}_{d = 0}{\rm  Ind}^{\tilde{G}}_{\tilde{B}}(\overline{\mathbb{F}}^{\chi^d}_p)$
\item $ {\rm Ind}^{\Gamma^1_{1, [\mathfrak{b}_i]}(\mathfrak{n})}_{\Gamma^1_{ 1, [\mathfrak{b}_i]}(\mathfrak{ pn})}(\overline{\mathbb{F}}_p) \cong \oplus^{ p^2 -2}_{d = 0}{\rm  Ind}^{\tilde{S}}_{\tilde{B}\cap \tilde{S}}(\overline{\mathbb{F}}^{\chi^d}_p).$
\end{enumerate}
\end{proposition}
\begin{proof}
Because of the bijection $ \Gamma_{1, [\mathfrak{b}_i]}(\mathfrak{pn}) \backslash \Gamma_{1, [ \mathfrak{b}_i]}(\mathfrak{n}) \longleftrightarrow  \tilde{U}\backslash \tilde{G}$ given by reducing modulo  $p,$  the transitivity of    $ {\rm Ind},$ and the observation above,  we have the following identifications of left  $\Gamma_{1 ,[\mathfrak{ b}_i]}(\mathfrak{n})$-modules:
\begin{eqnarray*}
{\rm Ind}^{\Gamma_{1, [\mathfrak{b}_i]}(\mathfrak{n})}_{\Gamma_{1, [\mathfrak{b}_i]}(\mathfrak{ pn})}  (\overline{\mathbb{F}}_p)  & \cong &   {\rm  Ind}^{\tilde{G}}_{\tilde{U}}(\overline{\mathbb{F}}_p) \\ 
& \cong  & {\rm  Ind}^{\tilde{G}}_{\tilde{B}}({\rm  Ind}^{\tilde{B}}_{\tilde{U}}( \overline{\mathbb{F}}_p)) \\ 
&\cong  & {\rm  Ind}^{\tilde{G}}_{\tilde{B}}( \oplus^{ p^2 -2}_{d = 0} (\overline{\mathbb{F}}^{\chi^d}_p)) \\
&\cong  &  \oplus^{p^2 -2}_{d = 0}{\rm  Ind}^{\tilde{G}}_{\tilde{B}}(\overline{\mathbb{F}}^{\chi^d}_p).
\end{eqnarray*}
For the second item, one uses the bijection $ \Gamma^1_{1, [\mathfrak{b}_i]}(\mathfrak{pn}) \backslash \Gamma^1_{1, [ \mathfrak{b}_i]}(\mathfrak{n}) \longleftrightarrow  (\tilde{U} \cap \tilde{S})\backslash \tilde{S}.$
\end{proof}
We shall need a more explicit version of ${\rm  Ind}^{\tilde{G}}_{\tilde{B}}(\overline{\mathbb{F}}^{\chi^d}_p).$
For  $ 0 \leq d \leq p^2 -2,$ we define the following  $ \tilde{G}$-module which we denote by  $  {\rm  U}_d(  \overline{\mathbb{F}}_p)\index{${\rm  U}_d(  \overline{\mathbb{F}}_p)$}:$  $$ {\rm U}_d(  \overline{\mathbb{F}}_p)  =\{  f :  \mathbb{F}^2_{p^2} \rightarrow \overline{\mathbb{F}}_p:   f( xa, xb) = x^d f( a, b) \, \forall \, x \in  \mathbb{F}^*_{ p^2} \}.$$
Next we  define the following homomorphism
\begin{eqnarray*} \varphi :  {\rm U}_d(\overline{\mathbb{F}}_{p})   & \rightarrow &  {\rm Ind }^{ \tilde{ G}}_{ \tilde{ B}} ( \overline{\mathbb{F}}^{ \chi^d}_{p})  \\
 F   & \mapsto &(\left(\begin{smallmatrix} a&b\\ c&e \end{smallmatrix}\right) \mapsto F(c, e)).
\end{eqnarray*}
We shall show that  it  is an isomorphism of $\tilde{G}$-modules.
It is well defined since \[\varphi( F )(\left( \begin{smallmatrix} x&y \\ 0& z \end{smallmatrix}\right) \left( \begin{smallmatrix} a&b \\ c&e \end{smallmatrix}\right)) =  F( zc, z e ) =  z^d F( c, e) = \chi^d( \left( \begin{smallmatrix} x&y \\ 0& z \end{smallmatrix}\right)) \varphi( F)( \left( \begin{smallmatrix} a&b \\ c&e \end{smallmatrix}\right)).\] 
It is also easy to see that  $\varphi$ is an $ \tilde{G}$-homomorphism. In order to conclude that $\varphi$ is an isomorphism,  one can  define the inverse $ \psi$ of $  \varphi$ as follows. We first note that for  $ c, \,  e \in \mathbb{ F}_{ p^2}$ not both zero we can find  $a , b \in \mathbb{ F}_{ p^2} $ such that $ a e - bc \neq 0.$ Hence an element $ (c, e) \neq (0,   0)$ gives rise to a matrix  $ \left(\begin{smallmatrix} a&b \\ c&e \end{smallmatrix}\right)$ in $ \tilde{ G}.$ Another choice of $ a', b'$ with $ a'e - b' c \neq 0 $ amounts to multiply  $ \left(\begin{smallmatrix} a& b \\ c& e \end{smallmatrix}\right) $  from the left by a matrix of the form $ \left(\begin{smallmatrix} 1& * \\ 0 & 1  \end{smallmatrix}\right)$ which acts trivially on $ \overline{\mathbb{ F}}^{ \chi^d }_{ p}.$ This implies that the map \label{FirstId}
\begin{eqnarray*} \psi:  {\rm Ind }^{\tilde{G}}_{\tilde { B}} ( \overline{\mathbb{F}}^{ \chi^d}_{p})  & \rightarrow &  {\rm U }_d ( \overline{\mathbb{F}}_{p}) \\
 f  & \mapsto & ((c , e) \mapsto f( \left( \begin{smallmatrix} a& b \\ c & e \end{smallmatrix}\right)),  
\end{eqnarray*} is well defined,  that is,  to mean that any choice of $ a, b \in \mathbb{ F}_{ p^2} $ with $ ae - bc \neq 0$ will do. Furthermore it is easy to verify that it is an $ \tilde{G}$-homomorphism and it is the inverse of $ \varphi.$ 

In the next remark, there is another proof of the isomorphism of $ \tilde{G}$-modules $:{\rm  Ind}^{\tilde{G}}_{\tilde{B}} (\overline{\mathbb{F}}^{\chi^d}_p) \cong {\rm U}_d(\overline{\mathbb{F}}_p).$
\begin{remark} 
We start with the identification of $ \overline{\mathbb{F}}_{p}$-vector spaces \[\overline{\mathbb{F}}_{p}[ X, Y]/ ( X^{ p^2} - X, Y^{ p^2} - Y) \cong \{ f : \mathbb{F}_{ p^2}^2 \rightarrow \overline{\mathbb{F}}_{p} \}\] where $ P( X, Y)$ maps to the function  $ ( a, b) \mapsto P( a, b).$ To see this we observe that the spaces on both sides have dimensions $ p^4  $ as $ \overline{\mathbb{F}}_{p}$-vector spaces. So, we just have to prove injectivity. To this end for any $x \in \mathbb{F}_{p^2} $ if the polynomial $ f_x( Y)  = P(x, Y  ) \in \overline{\mathbb{F}}_{p}[ Y]$  vanishes for all $ y \in \mathbb{ F}_{ p^2},$ then this means that $  Y $  and $ Y^{ p^2 - 1} - 1 $ divide $ f_x( Y)$ for all $ x \in \mathbb{ F}_{ p^2}.$ Because $x$ is arbitrarily chosen we deduce that $ Y^{ p^2} -Y $ divides $ P( X, Y).$   As the role of $ X$ and $ Y$ are symmetric, one obtains that $ P(X, Y)$ lies in the ideal $ ( X^{ p^2} - X, Y^{p^2} - Y).$ In fact this is  an isomorphism of $ \overline{\mathbb{F}}_{p}[\tilde{G}]$-modules.
Let  $ \mathcal{ W}(p, \overline{\mathbb{F}}_{p}) := \{ f \in \overline{\mathbb{F}}_{p}[X, Y]/ ( X^{ p^2} - X, Y^{ p^2} - Y) : f( ( 0, 0)) = 0\}.$ This module can be identified with ${\rm Ind}^{\Gamma_{1, [\mathfrak{b}_i]}(\mathfrak{n})}_{\Gamma_{ 1, [\mathfrak{b}_i]}(\mathfrak{ pn})}(\overline{\mathbb{F}}_p)$ as $\Gamma_{1, [\mathfrak{b}_i]}(\mathfrak{n})$-module, see \cite{Gabor} for details.
Then $  {\rm U}_d ( \overline{\mathbb{F}}_{p})$ is the subspace of homogeneous polynomial classes of degree $d$ in $ \overline{\mathbb{F}}_{p}[ X, Y]/ ( X^{ p^2} - X, Y^{ p^2} - Y)$ with $  f( (0, 0)) = 0$.
And  as  a graded $\Gamma_{1, [\mathfrak{b}_i]}(\mathfrak{n})$-module $\mathcal{W}(p,  \overline{\mathbb{F}}_{p}) $ decomposes as follows: \[{\rm \mathcal{W}}(p, \overline{\mathbb{F}}_{p}) = \oplus^{ p^2 -2}_{ d = 0} {\rm U }_d( \overline{\mathbb{F}}_{p}).\]
\end{remark}
The  isomorphism  in Proposition \ref{Prop2-2} will permit us to obtain a better understanding of the non-semisimple $\tilde{G}$-module $ {\rm U }_{d}(\overline{\mathbb{F}}_{p}).$ We shall turn to this among other things.
\section{Irreducible $ \tilde{G}$-modules}\label{sec2_4}
We keep the same notation as in the previous sections. We will prove here the main results.   For an irreducible $ \overline{\mathbb{F}}_{p}[{\tilde{G}}]$-module $W,$ this is done by embedding a cohomology group with coefficients in $ W$ into another cohomology group with trivial coefficients roughly speaking.

First of all, we shall see how the irreducible $ \overline{\mathbb{F}}_{p}[ \tilde{G}]$-modules 
can be embedded in a twist of  $ {\rm U}_d( \overline{\mathbb{F}}_p).$ Let $ \tau$ be the non-trivial automorphism of $ \mathbb{ F}_{p^2}.$ For $ 0 \leq r,  s \leq  p-1,  0\leq l, t, \leq p-1,$ recall that when $ l$ and $t$ are not both equal to $ p-1,$
the representations $$ V^{l, t}_{r, s} (\overline{\mathbb{F}}_{p})  :=  {\rm Sym}^r (\overline{\mathbb{ F}}^2_{p}) \otimes_{\mathbb{F}_{p^2}} det^l \otimes_{ \mathbb{F}_{p^2}}  {\rm  Sym}^s (\overline{\mathbb{ F}}^2_{p})^{\tau}\otimes_{ \mathbb{F}_{p^2}} (det^t)^\tau,$$ exhaust all the irreducible $\overline{\mathbb{F}}_{p} [  \tilde{G} ]$-modules.  Here, we identify $ {\rm Sym}^r( \overline{\mathbb{ F}}^2_{ p})$ with the homogeneous polynomials in the variables $ X, Y$ over $ \overline{\mathbb{F}}_{p}$ of degree $ r$ which we denote by $ \overline{\mathbb{F}}_{p}[ X,  Y]_r.$   
A matrix  $  \left( \begin{smallmatrix} a&b\\ c&d \end{smallmatrix}\right) \in \tilde{G}$ acts from the left on $  V^{l, t}_{ r, s}(\overline{\mathbb{F}}_{p})$  as follows: on the first factor $ \left( \begin{smallmatrix} a&b\\ c&d \end{smallmatrix}\right). X^ iY^j := ( aX + b Y)^i( cX + d Y )^j $ followed by multiplication by  $ (ad -bc)^l$ and on the second factor we apply first $ \tau$ on $ \left( \begin{smallmatrix} a&b\\ c&d \end{smallmatrix}\right)$ and proceed as for the first factor followed by multiplication by $ (ad-bc)^{pt}, $ e.g, $$  \left(\begin{smallmatrix} a&b\\ c&d \end{smallmatrix}\right).X^iY^j\otimes X^{ i'} Y^{j'}:= (ad -bc )^{ l+ pt }(aX + b Y)^i( cX + d Y )^j\otimes ( a^pX + b^p Y)^{i'}( c^pX + d^p Y )^{j'}.$$
For $  e \geq 0, $  we write  $ {\rm U }^e_d( \overline{\mathbb{F}}_{p}) $ to mean  the $ \overline{\mathbb{F}}_p[\tilde{G}]$-module $ {\rm  U}_d (\overline{\mathbb{F}}_{p}) $ with the natural action of $ \tilde{G}$ followed by multiplication by $det^e, $ e.g,  ${\rm U }^e_d( \overline{\mathbb{F}}_{p})\index{${\rm U }^e_d( \overline{\mathbb{F}}_{p})$} =  {\rm U }_d( \overline{\mathbb{F}}_{p})\otimes_{ \mathbb{ F }_{ p^2}} det^e.$
\begin{lemma}\label{emblem}
We have the following embedding of left $ \overline{\mathbb{F}}_{p}[ \tilde{G}]$-modules 
\begin{eqnarray*}\label{equ1} \Psi: V^{q, t}_{ r,  s}( \overline{\mathbb{F}}_{p})  &\rightarrow &   {\rm U}^{ q +pt}_{r  + ps}(\overline{\mathbb{F}}_{p}) \\ 
 f\otimes g  & \mapsto & ((  a, b) \mapsto f(a, b) g( a^p, b^p )).
\end{eqnarray*}
\end{lemma}
\begin{proof}
In polynomial terms we can write  $\Psi ( f(X, Y) \otimes g(X, Y) ) = f( X, Y)g( X^p, Y^p).$
By definition of $ \Psi$ we have $ \Psi ( \sum f_i\otimes g_i )=\sum \Psi( f_i\otimes g_i).$ Now let  $M = \left( \begin{smallmatrix} a&b\\ c&d \end{smallmatrix}\right)  \in \tilde{G},$ we need to check that $ \Psi(  M.f\otimes M^{ \tau}. g )  =  M.\Psi(  f\otimes g).$ 
For $ f( X, Y) =\sum_{ n+ m = r } a_{ n, m} X^nY^m,$ $ g( X, Y) = \sum_{ l+ k = s} b_{l, k}X^lY^k,$ then $ M.f  = (ad -bc )^q\sum_{n+m = r} a_{n, m} (aX+ bY)^n( cX + dY )^m$ and $ M^{ \tau}.g = (ad -bc )^{ pt}\sum_{ l+k = s}b_{l,k}( a^pX +  b^pY )^l( c^pX + d^pY)^k.$
We set  $\alpha = (ad -bc )^{ q+ pt}.$ Hence by denoting $ \Psi (M.( f\otimes g ))$ as $(*),$ we have
\begin{eqnarray*}
 (*)& = &\alpha\sum_{ n+m = r } \sum_{ l+k = s} a_{ n, m}( aX +   bY )^n( cX + dY )^m b_{ l, k}( a^{ p} X^p + b^{p}Y^p)^l  (c^{ p}X^p +  d^{ p}Y^p)^k \\
&= & \alpha\sum_{ n+m = r }\sum_{l+k = s}  a_{ n, m}( aX + bY )^n( cX + dY )^m b_{ l, k}( aX + bY)^{pl} ( cX + dY)^{pk}	
\\  &= & M.\Psi ( f\otimes g).
\end{eqnarray*}
\end{proof}
One would then like to have a more concrete description of the cokernel of $\Psi$. In other words, one has to compute the Jordan-H\"{o}lder series of  $  {\rm U}^e_{ r +ps}( \overline{\mathbb{F}}_p).$
\begin{remark}
In the special case $  s = 0,$ the semi-simplification of $ {\rm U }^e_{r}(  k ) $ for $ k$  a finite field can be obtained by immediate generalization of the case $ k = \mathbb{ F}_p$ which is treated for instance in \cite{Gabor}. But as it seems that this method does not apply when $ s > 0,$ we will naturally follow the  Brauer character theory approach which gives the semisimplification of $ {\rm U }^e_{ d}( k  )$ in complete generality.
\end{remark}
For our purpose we shall next see that 
\[( {\rm U }^{}_{r + ps} (\overline{\mathbb{F}}_{p} ))^{ss} = V^{0, 0}_{r, s}( \overline{\mathbb{F}}_p )\oplus V^{r,s}_{ p- r -1, p-1- s}( \overline{\mathbb{F}}_p )\oplus V^{0, s+1}_{r-1, p-2 - s}( \overline{\mathbb{F}}_p ) \oplus V^{r+1,0}_{p -r- 2, s-1}( \overline{\mathbb{F}}_p ).\]
From this we deduce the semisimplification of $  {\rm U }^{e}_{r + ps}(\overline{\mathbb{F}}_p)$ by twisting.
\subsection{The constituents of  $ {\rm U }^{e}_{r + ps}(\overline{\mathbb{F}}_p)$}
Let $ k$ be a  finite field and $ \mathfrak{G} = {\rm GL}_2(k), \,\textrm{and}\, \mathfrak{B}$ its Borel subgroup of upper triangular matrices. For a character $ \phi$ of $ \mathfrak{B}$ with values in  $\overline{\mathbb{F}}_p,$ we consider  ${\rm Ind }^{\mathfrak{G}}_{\mathfrak{B}}( \overline{\mathbb{F}}^{\phi^d}_p)$ where  $ 0 \leq\ d  \leq \sharp k -2.$ The semisimplication of $ {\rm Ind }^{\mathfrak{G}}_{\mathfrak{B}}( \overline{\mathbb{F}}^{ \phi^d}_p)$  is computed  in \cite{Diamond} via Brauer character theory. Given two homomorphisms $ \chi_1, \,\chi_2 :  k^* \rightarrow \overline{\mathbb{ Q }}^* (\,  \textrm{or}\,\,  \overline{\mathbb{ Q}}^*_p ),$ one obtains a character of $ \mathfrak{B}$ induced by $ \chi_1,  \, \chi_2$ as \[ \left(\begin{smallmatrix} a&b \\ 0& e \end{smallmatrix}\right) \mapsto \chi_1( a)\chi_2( e).\]
Furthermore for $  V = \overline{\mathbb{Q}}\,\, (\,\,\textrm{ or}\,\, \overline{\mathbb{ Q}}_p)$ let $ {\rm I}( \chi_1, \chi_2) := {\rm Ind}^{\mathfrak{G}}_{\mathfrak{B}} ( V^{\chi_1, \chi_2})$ where $${\rm Ind}^{\mathfrak{G}}_{\mathfrak{B}} ( V^{\chi_1, \chi_2})  = \{f : \mathfrak{G} \rightarrow V: f (  \left(\begin{smallmatrix} a&b \\ 0& e \end{smallmatrix}\right) g) = \chi_1(a)\chi_2( e) f( g)\,\, \forall \left(\begin{smallmatrix} a&b \\ 0& e \end{smallmatrix}\right) \in \mathfrak{B},\, g \in \mathfrak{G}\}. $$ This is a $ ( q+1)$-dimensional representation of $ \tilde{G}$ where $ q = \sharp k.$ It is known as  a principal series representation of $ \mathfrak{G}.$  Next let $ E $ be the set of embeddings $ k  \rightarrow \overline{{\mathbb{ F}}}_p$.  Then the complete list of irreducible  $ \overline{\mathbb{F}}_p$-representations of $\mathfrak{G}$ is given by:\[
R_{\overrightarrow{m}, \overrightarrow{n}}  =   \otimes_{ \tau\in E}  ({\rm Sym }^{n_{ \tau} -1}  (k^2) ))^{\tau} \otimes (det^{m_{ \tau}})^\tau \otimes \overline{{\mathbb{F}}}_{ p};\]for integers $ 0 \leq m_{ \tau} \leq  p-1 $ and $  1\leq n_{\tau} \leq p $ associated with each $ \tau \in E,$ and some $ n_\tau$ is less than $ p-1.$ Here one makes  the convention  $ {\rm Sym^{-1} } ( k^2) = \{  0\}, $ the null module. Before we go further, note that in our notation we have $$ V^{l, t}_{r, s}(\overline{\mathbb{F}}_p)  =  R_{ (l, t), ( r + 1, s+ 1)}$$ as irreducible $ \mathfrak{G}$-modules.

To obtain the semisimplification of an  $\overline{\mathbb{ F}}_p$-representation of $ \mathfrak{G},$ the approach is via Brauer character theory. One starts with a $\overline{\mathbb{Q}}_p$-representation $ W$ of  $ \mathfrak{G},$ and reduction modulo the maximal ideal of $\overline{\mathbb{Z}}_p$ yields an $ \overline{\mathbb{ F}}_p$-representation of $ \mathfrak{G}.$ More precisely for such a $ W,$ we know that there exists a $ \overline{\mathbb{Z}}_p $-lattice $ L$ inside $ W$ invariant under the action of $ \mathfrak{G}.$ Then  reduction of $ L$ modulo the maximal ideal of $ \overline{\mathbb{Z}}_p$ gives rise to an $ \overline{ \mathbb{F}}_p$-representation whose Brauer character is the restriction of the character of $ W$ to the $ p$-regular classes of $ \mathfrak{G}.$ In this way the semisimplification thus obtained is independent of the lattice $ L.$

Any group homomorphism  $ \varphi :k^* \rightarrow \overline{\mathbb{Q}}^*_p $ can be written as $\varphi = \prod_{ \tau} \tilde{ \tau}^{ a_{ \tau}}$ with $ 0\leq a_{\tau} \leq p-1 $  and $\tilde{\tau}$ the Teichm\"{u}ller lift of $ \tau.$ Then the  reduction of $\varphi$ is $ \bar{ \varphi} = \prod_{\tau} \tau^{a_{\tau}}.$ By a twist it suffices to consider the irreducible representation of the form $ {\rm I }( 1, \chi).$ Then  it was shown in \cite{Diamond} that
\begin{proposition}[Diamond]\label{propDia}
Let $   M =  {\rm  I} (1 ,\prod_{\tau} \tilde{ \tau}^{ a_{ \tau}})$ with $ 0 \leq a_{\tau} \leq p-1 $ for each $  \tau \in E.$  Let $Frob$ be the Frobenius in $ E.$ Then the semisimplification of the reduction $ \overline{M}$ of $ M$ is $ \overline{M } \cong \oplus_{ J \subset S } M_J$ with $ M_J  = R_{\overrightarrow{m}_J, \overrightarrow{ n}_J },$ where 
\begin{eqnarray*}
m_{J , \tau} = \begin{cases}  0 \,\,\,\,\, \textrm{if} \,\,\,\,\tau \in  J \\
a_{\tau} + \delta_{ J}(\tau)\,\,\,\,\, \textrm{otherwise};
\end{cases}
n_{J, \tau} = \begin{cases}  a_{\tau} + \delta_{ J}( \tau )\,\, \,\,\,\, if \,\,\,\,\, \tau \in J \\
 p - a_{\tau} - \delta_{J} (\tau )\,\, \,\, \,otherwise;
\end{cases}\\
with \,\,\delta_{ J}\, \,the \,\, characteristic \,\, function \,\, of \,\,  J^{ ( p) } = \{ \tau\circ Frob: \tau \in J\}.
\end{eqnarray*}
\end{proposition}
We shall next specialize to our setting. We take  $\mathfrak{G} = \tilde{G}$ and  $\mathfrak{B} = \tilde{B}.$
Let  $\overline{\psi}:\mathbb{F}_{ p^2} \hookrightarrow \overline{ \mathbb{ F}}_p $ be a fixed injection. We also denote by $ \overline{\psi}$ the character obtained by restriction to $ \mathbb{ F}^*_{ p^2}$.
Now, let $\psi :  \mathbb{F}^*_{ p^2} \rightarrow \overline{\mathbb{ Q}}^*_p$  be the Teichm\"{u}ller lift of $ \overline{\psi}.$ The homomorphism $\psi$ induces the character of  $\tilde{B}:$ \[\tilde{ B} \rightarrow \overline{\mathbb{Q}}^*_p;\, \,  \,  \left(\begin{smallmatrix} x&y \\ 0&z \end{smallmatrix}\right)
\mapsto \psi( z).\]
Via $\psi,$ $\overline{\mathbb{Q}}_p$  is endowed with a structure of $\tilde{ B}$-module which we denote $ \overline{\mathbb{Q}}_p^{\psi}:$ for  $ h \in \tilde{B},  x  \in \overline{\mathbb{Q}}_p$ we have  $ h.x = \psi(h)x.$ The  $\overline{\mathbb{Q}}_p$-representation $$ {\rm I } (1, \psi) = {\rm Ind }^{\tilde{ G}}_{\tilde{B}} ( \overline{\mathbb{Q}}_p^{ \psi}) = \{ f:  \tilde{G} \rightarrow \overline{\mathbb{Q}}_p : f( hg ) = \psi(h) f( g) \,\, \forall h\in \tilde{ B},\, g \in \tilde{ G} \}$$ of $ \tilde{G} $ has reduction the $ \overline{\mathbb{ F}}_p$-representation $ {\rm  Ind}^{ \tilde{ G}}_{ \tilde{ B}} ( \overline{\mathbb{F}}_p^{ \overline{\psi}})$ 
where by abuse of notation $\overline{\psi}$ is the character of $\tilde{B}$  induced by $ \overline{\psi}:$   \[\tilde{B} \rightarrow \overline{{\mathbb{ F}}}^*_p; \,\, \left(\begin{smallmatrix}  x&y \\ 0& z \end{smallmatrix}\right)\mapsto \overline{\psi( z)}.\]
For $ 0 \leq d\leq p^2-2,$ we consider the $\overline{\mathbb{F}}_p$-representations of  $\tilde{G}$ : $ {\rm Ind }^{ \tilde{G}}_{\tilde{B}}(\overline{\mathbb{F}}^{\overline{\psi}^d}_{p})\index{$ {\rm Ind }^{ \tilde{G}}_{\tilde{B}}(\overline{\mathbb{F}}^{\overline{\psi}^d}_{p})$}.$ We write  $d =  r + ps$ with $ 0 \leq r , s \leq  p-1,$   $E  = \{ id, \tau  \}$ so that  $ \overline{\psi }^d  =  id^r \tau^s.$  From Proposition \ref{propDia},  we have \[({\rm U}^{}_{r + ps} (\overline{\mathbb{F}}_{p} ))^{ss} = V^{0, 0}_{r, s}( \overline{\mathbb{F}}_p )\oplus V^{r,s}_{p- r -1, p-1- s}( \overline{\mathbb{F}}_p )\oplus V^{0, s+1}_{r-1, p-2 - s}( \overline{\mathbb{F}}_p ) \oplus V^{r+1,0}_{p -r- 2, s-1}( \overline{\mathbb{F}}_p );\]
where we have used the identification of $ {\rm U}_d( \overline{\mathbb{F}}_{p}) $ with $ {\rm Ind }^{ \tilde{ G}}_{ \tilde{ B}} ( \overline{\mathbb{F}}^{\overline{\psi}^d}_{p})$ as $ \overline{\mathbb{F}}_{p}[\tilde{G}]$-modules from page \pageref{FirstId}.
We define the representation $ W^{l, t}_{r, s}$ by the exact sequence\[ 
0\rightarrow V^{l, t}_{r, s}( \overline{\mathbb{F}}_p) \rightarrow {\rm U}^{ l+ pt }_{ r + ps} ( \overline{\mathbb{ F}}_{p}) \rightarrow W^{l, t}_{r, s} \rightarrow 0.\] Thus the semisimplification of $ W^{l, t}_{r, s}$ is $$ (W^{l, t}_{r, s})^{ss} = V^{r+l, s+t}_{p- r- 1, p-s- 1}( \overline{\mathbb{F}}_p) \oplus V^{ l, s+1+t}_{r-1,  p- s- 2 }( \overline{\mathbb{F}}_p)\oplus  V^{r+l+1, t}_{p - r -2, s- 1}(\overline{\mathbb{F}}_p).$$
\subsection{Some invariants}
Let $\Gamma_{1, [\mathfrak{b}_i]}(\mathfrak{n})$ be the congruence subgroups of $ G(F) $ defined in Section \ref{sect2_2}. We view $ \overline{\mathbb{F}}_{p}$ as a trivial left $ \tilde{G}$-module. We need to remind us once more how  Hecke operators act on the degree zero group cohomology. Let  $ g \in \Delta^\mathfrak{q}_1(\mathfrak{n}) $ where $  \Delta^\mathfrak{q}_1(\mathfrak{n}) $ is the subset of $ Mat_2(\hat{\mathcal{O}})_{\neq 0}$ defined in Section \ref{sect2_2}.
From Lemma \ref{lem2_3_0}, we have that for each  $ i,   1 \leq  i \leq h,$ there are  a unique index $ j_i$ and   a matrix  $ \beta_i \in \Lambda^\mathfrak{c}_{1, [\mathfrak{b}_{i}]}(\mathfrak{n})$  such that $  K_1(\mathfrak{n}) g K_1(\mathfrak{n}) = K_1(\mathfrak{n}) g_{j_i}\beta_i g^{-1}_{i} K_1(\mathfrak{n}) $  with $ g_i$ the matrices corresponding to the ideal classes as defined in Subsection \ref{Subsection2_2_2}. 

Let  $ M$ be a finite dimensional left $ \overline{\mathbb{F}}_p[\tilde{G}]$-module.  We have seen that the Hecke operator corresponding to the double coset $ K_1(\mathfrak{n})gK_1(\mathfrak{n})$ which we have denoted as $T_g$ sends $ (m_1, \cdots, m_h) \in \oplus^h_{ i = 1}{\rm H}^0(\Gamma_{1, [\mathfrak{ b}_i]}(\mathfrak{n}), M) $ to   $ (n_1, \cdots, n_h)$ with  $ n_{j_i} = T_{\beta_i}.m_i.$ Here $ T_{\beta_i}$ is the Hecke operator corresponding to the double coset $\Gamma_{1, [\mathfrak{b}_{j_i}]}(\mathfrak{n})\beta_i \Gamma_{1, [\mathfrak{b}_{i}]}(\mathfrak{n}).$
Explicitly one defines $ \Gamma''^{, \beta_i}_{1, [\mathfrak{b}_{j_i} ]}(\mathfrak{n}) = \beta_i\Gamma_{1, [\mathfrak{b}_i]}(\mathfrak{n})\beta^{-1}_i\cap \Gamma_{1, [\mathfrak{b}_{j_i}]}( \mathfrak{n}), $ then we have
\begin{eqnarray*}
T_{\beta_{i}}: {\rm H}^0(\Gamma_{1, [\mathfrak{b}_i]}(\mathfrak{n}), M)  & \rightarrow& {\rm H}^0(\Gamma_{1, [\mathfrak{b}_{j_i}]}(\mathfrak{n}), M) \\
m  &  \mapsto & \sum_{\lambda\in\Gamma_{1, [\mathfrak{b}_{j_i}]}(\mathfrak{n})/ \Gamma''^{, \beta_i}_{1, [\mathfrak{b}_{j_i}]}(\mathfrak{n})} (\lambda\beta_i).m.
\end{eqnarray*}
This being said here are the $ \Gamma^1_{1, [\mathfrak{b}_i]}(\mathfrak{n})$ and $ \Gamma_{1, [\mathfrak{b}_i]}(\mathfrak{n})$-invariants for $  {\rm U }^e_d( \overline{\mathbb{F}}_{p}),$  $ V^{l, t}_{r, s}(\overline{\mathbb{F}}_p),$   $ (W^{l, t}_{r ,s})^{ ss}$ and $  W^{l, t}_{ r, s}.$
\begin{lemma}\label{lem2}
Let  $ d$ and $ n$ be integers greater than or equal to zero. Then  one has
\begin{enumerate}
\item for all $ n \geq 0, $  one has
$$ \oplus^h_{ i = 1}{\rm H }^0( \Gamma^1_{1, [\mathfrak{b}_i]}(\mathfrak{n}),  {\rm U}^n_d( \overline{\mathbb{F}}_{p}))   = \left\{ \begin{array}{rcc}
\oplus^h_{i = 1}\overline{\mathbb{F}}_{p}&  \text{if} \,\,
    d \equiv     0 \pmod{p^2 -1}  \\
0&   \text{otherwise}
\end{array} \right.$$
as $\overline{\mathbb{F}}_p$-vector spaces. 
\item
$ \oplus^h_{ i = 1}{\rm H }^0( \Gamma_{1, [\mathfrak{b}_i]}(\mathfrak{n}),  {\rm U}^n_d( \overline{\mathbb{F}}_{p}))   = \left\{ \begin{array}{rcc}
\oplus^h_{i = 1}\overline{\mathbb{F}}_{p}& \,\,\text{if} \
    d \equiv 0 \pmod{p^2 -1} \, \text{and} \, (\mathcal{O}^*)^n = 1  \\
0  & \text{otherwise} 
\end{array} \right.$ \\
as $\overline{\mathbb{F}}_p$-vector spaces.
\item the Hecke operator $ T_g$  acts on  $ (m_1, \cdots, m_h) \in\oplus^h_{ i = 1}{\rm H }^0( \Gamma_{1, [\mathfrak{b}_i]}(\mathfrak{n}),  {\rm U}^n_d( \overline{\mathbb{F}}_{p})) =  \oplus^h_{i = 1}\overline{\mathbb{F}}_{p},$  where $ d \equiv 0 \pmod{p^2-1}$ and $ (\mathcal{O}^*)^n = 1,$   by sending  $  m_i$  to $ n_{j_i}$ with   $$  n_{ j_i} = [{\Gamma_{1, [\mathfrak{b}_{j_i}]}(\mathfrak{n})}: \beta_{i}{\Gamma_{1, [\mathfrak{b}_{i}]}(\mathfrak{n})}\beta^{ -1}_i\cap {\Gamma_{1, [\mathfrak{b}_{j_i}]}(\mathfrak{n})}] m_i. $$
\end{enumerate}
\end{lemma}
\begin{proof}
As for the first item, because $ \Gamma^1_{1, [\mathfrak{b}_i]}(\mathfrak{n})$ reduces modulo $\mathfrak{p}$ to $ \tilde{S}$ and since
we can identify  $ {\rm U }^n_d( \overline{\mathbb{F}}_p)$ with $ {\rm U}_d( \overline{\mathbb{F}}_p)$ as  $\tilde{S}$-module,   we see that the invariants do not depend on the values of  $ n.$ Having this, 
it is suitable to view ${\rm U}_d(\overline{\mathbb{ F}}_{p})$ as the set of $ \overline{\mathbb{ F}}_{p} $-valued functions on  $\mathbb{F}^2_{p^2}$  and  homogeneous of degree $ d.$ 
Observe first that a non-null constant function belongs to  $ ({\rm U }_d ({\overline{\mathbb{F}}_{p}} ))^{\Gamma^1_{1, [\mathfrak{b}_i]}(\mathfrak{n})}$ if and only if $ d \equiv 0  \pmod{p^2 -1}.$ Any nonzero $ f \in ({\rm U }_d ({\overline{\mathbb{F}}_{p}} ))^{\Gamma^1_{1, [\mathfrak{b}_i]}(\mathfrak{n})}$ is a constant function. Indeed let $ ( 0 , 0 ) \neq  ( a, b), ( a', b') \in \mathbb{ F}^2_{p^2} $ and  suppose that $ f( a, b ) = x   \neq f( a', b') = y.$ Now since  $ ( a, b),\, ( a', b')\neq( 0, 0)$ there are $ c, e, \, c', \, d' \in \mathbb{ F}_{ p^2}  $  such that  $\left(\begin{smallmatrix} a& b \\ c & e \end{smallmatrix}\right), \, \, \left(\begin{smallmatrix} a'&b'\\ c'& d'\end{smallmatrix}\right) \in \tilde{S}.$ Then $( a', b') = (a, b) \left(\begin{smallmatrix} e&-b \\ -c& a \end{smallmatrix}\right) \left(\begin{smallmatrix} a'&b'\\ c'& d'\end{smallmatrix}\right).$  Therefore $ \Gamma^1_{1, [\mathfrak{b}_i]}(\mathfrak{n}) $ acts transitively on $ \mathbb{F}^2_{p^2}-\{(0 , 0)\}.$ Indeed from Lemma \ref{lem1},  reduction modulo $ \mathfrak{p}$ is a surjective homomorphism $ \Gamma^1_{1, [\mathfrak{b}_i]}(\mathfrak{n})  \twoheadrightarrow  \tilde{S}.$  So we have  $y =  f( a', b') = f( (a, b)\gamma) =  \gamma f( (a, b)) = f( a, b ) =  x, $ contradicting the hypothesis $ x \neq y.$  Hence  $ f \in ({\rm U }_d ({\overline{\mathbb{F}}_{p}} ))^{\Gamma^1_{1, [\mathfrak{b}_i]}(\mathfrak{n})}$ if and only if  $ f$ is constant.\\
For the second item one firstly observes  that a non-null constant function belongs to  $ ({\rm U }^n_d ({\overline{\mathbb{F}}_{p}} ))^{\Gamma_{1, [\mathfrak{b}_i]}(\mathfrak{n})}$ if and only if $ d \equiv 0  \pmod{p^2 -1}$ and $ (\mathcal{O}^*)^n = 1.$ From here the same proof as the one given for the first item applies.
\\For the third item, let $ f_x \in ({\rm U}^n_d ({\overline{\mathbb{F}}_{p}}))^{\Gamma_{1, [\mathfrak{b}_i]}(\mathfrak{n})}$ with  $ f(a, b) = x $ for all   $ (a, b) \in \mathbb{ F}^2_{ p^2}-\{(0,  0)\}$ and let given  $ \Gamma_{1, [\mathfrak{b}_{j_i}]}(\mathfrak{n})
\beta_i\Gamma_{1, [\mathfrak{b}_i]}(\mathfrak{n})=\amalg_k \delta_k\Gamma_{1, [\mathfrak{b}_{i}]}(\mathfrak{n}) .$ Then $$ T_{ \beta_i}.f_x( a, b ) = \sum \delta_k\beta_i.f_x( a, b) = [\Gamma_{1, [\mathfrak{b}_{j_i}]}(\mathfrak{n}): \beta_i\Gamma_{1, [\mathfrak{b}_i]}(\mathfrak{n})\beta^{-1}_i\cap\Gamma_{1, [\mathfrak{b}_{j_i}]}(\mathfrak{n})] x.$$ From this  what we have claimed follows.
\end{proof}
We also have the following 
\begin{lemma}\label{lem3} 
Let $   0\leq r, s\leq p-1,$  and let $   l, t$ be integers greater than or equal to zero. Then one has
\begin{enumerate} 
\item  $ \oplus^h_{ 1= i}{\rm H }^0( \Gamma^1_{1, [\mathfrak{b}_i]}(\mathfrak{n}), V^{l, t}_{r, s}(\overline{\mathbb{F}}_p)   )   = \left\{ \begin{array}{rcc}
\oplus^h_{1= i}\overline{\mathbb{F}}_{p}& \text{if}\,\,
    r = s=  0 \,\,\text{and for all}\,\,\, l, t  \\
0  & \text{otherwise} 
\end{array} \right.$ \\
as $ \overline{\mathbb{F}}_{p}$-vector spaces.
\item $ \oplus^h_{ 1= i}{\rm H }^0( \Gamma_{1, [\mathfrak{b}_i]}(\mathfrak{n}), V^{l, t}_{r, s}(\overline{\mathbb{F}}_p)   )   = \left\{ \begin{array}{rcc}
\oplus^h_{1= i}\overline{\mathbb{F}}_{p} &\text{if}\,\,
r = s=  0 \,\,\text{and}\,\,\, (\mathcal{O}^*)^{l +pt} = 1  \\
0  & \text{otherwise} 
\end{array} \right.$ \\
as $ \overline{\mathbb{F}}_{p}$-vector spaces.
\item the Hecke operator $ T_g$  acts on  $ (m_1, \cdots, m_h)$  from  $\oplus^h_{i = 1}{\rm H }^0(\Gamma_{1, [\mathfrak{b}_i]}(\mathfrak{n}),  V^{l, t}_{0, 0}( \overline{\mathbb{F}}_{p}))$   which is equal to    $\oplus^h_{i = 1}\overline{\mathbb{F}}_{p}$ when $ (\mathcal{O}^*)^{l +pt} = 1,$ by sending  $m_i$  to $ n_{j_i}$ with    $ n_{j_i} = [{\Gamma_{1, [\mathfrak{b}_{j_i}]}(\mathfrak{n})}: \beta_{i}{\Gamma_{1, [\mathfrak{b}_{i}]}(\mathfrak{n})}\beta^{-1}_i\cap {\Gamma_{1, [\mathfrak{b}_{j_i}]}(\mathfrak{n})}] m_i.$
\end{enumerate}
\end{lemma}
\begin{proof} 
Firstly  when $ r = s = 0, $ and for all  $ l , t$ then  $ V^{l, t}_{r, s}(\overline{\mathbb{F}}_p) = \overline{\mathbb{F}}_p $ as $ \tilde{S}$-module.  By definition we have  $ \overline{\mathbb{F}}^{\Gamma^1_{1, [\mathfrak{b}_i]}(\mathfrak{n})}_p = \overline{\mathbb{F}}_p.$ 
Otherwise use the fact that   $ V^{l, t}_{r, s}( \overline{\mathbb{F}}_p)$  is irreducible as $\Gamma^1_{1, [\mathfrak{b}_i]}(\mathfrak{n})$-module.\\
The second item is proved similarly.
The statement about the Hecke action is verified similarly as in the proof of Lemma \ref{lem2}.
\end{proof}
\begin{lemma}\label{lem3.6}
Let $   0\leq r, s\leq p-1,  e := l + tp, e_1 := e + p(p-1), e_2 := e +p-1  \geq 0$ and let  $ f $ be the order of $ \mathcal{O}^*.$
The following isomorphisms of   $\overline{\mathbb{F}}_p$-vectors spaces hold: 
\begin{enumerate}
\item $  \oplus^h_{ 1 = i}{\rm H }^0(\Gamma^1_{1, [\mathfrak{b}_i]}(\mathfrak{n}), (W^{ }_{r, s})^{\textrm{ss}})= 
\left\{ \begin{array}{rcc}
&\oplus^h_{1= i}\overline{\mathbb{F}}_{p}& \,\text{if} \begin{cases} r = s=  p-1 \,\,\,\text{or}\, \\
r = 1,  s=  p-2 \, \text{or}
\\
r = p-2,  s=  1 \, 
 \end{cases}
 \\
&0   &\text{otherwise}
\end{array} \right.$ 
\item suppose that  $(r \neq 1 \,\text{or}\,  s \neq p-2 )$ and  $ (r  \neq p-2\,\text{or}\, s \neq 1\,\ )$; then we have $$ \oplus^h_{1 = i}{\rm H}^0(\Gamma^1_{1, [\mathfrak{b}_i]}(\mathfrak{n}), W^{ }_{r, s}) = \left\{ \begin{array}{rcc}
\oplus^h_{1= i}\overline{\mathbb{F}}_{p}& \,\,\,\text{if}\,\,
    r = s=  p-1 \,\,  \\
0  & \text{otherwise}
\end{array} \right.$$
\item $  \oplus^h_{ 1 = i}{\rm H }^0(\Gamma_{1, [\mathfrak{b}_i]}(\mathfrak{n}), (W^{l, t}_{r, s})^{\textrm{ss}})   = 
\left\{ \begin{array}{rcc}
&\oplus^h_{1= i}\overline{\mathbb{F}}_{p}& \,\text{if} \begin{cases} r = s=  p-1 \,\text{and}\,\,f\mid e\,\text{or}\, \\
r = 1,  s=  p-2 \,\text{and} \,f\mid e_1\, \text{or}
\\
r = p-2,  s=  1 \,\,\text{and}\,\, f\mid e_2  
\end{cases}
\\
&0   &\text{otherwise}.
\end{array} \right.$ 
\item suppose that  $(r \neq 1\,\text{or}\,  s \neq p-2 \,\,\,\text{ or}\,\,\, f\nmid e_1 $) and  $ (  r  \neq p-2\,\text{or}\, s \neq 1\,\,\,\text{ or}\,\,\, f\nmid e_2  )$; then we have $$ \oplus^h_{1 = i}{\rm H}^0( \Gamma_{1, [\mathfrak{b}_i]}(\mathfrak{n}), W^{l, t}_{r, s}) = \left\{ \begin{array}{rcc}
\oplus^h_{1= i}\overline{\mathbb{F}}_{p}& \text{if}\,\,
    r = s=  p-1 \,\,\text{and}\,\,\,f\mid e  \\
0  & \text{otherwise}
\end{array} \right.$$ 
\end{enumerate}
Lastly, the Hecke action on these spaces is as in the previous lemmas.
\end{lemma}
\begin{proof}
We have $(W^{ }_{r, s})^{\textrm{ss}} = V^{r, s}_{p- r- 1, p-s- 1}(\overline{\mathbb{F}}_p) \oplus V^{ 0, s+1}_{r-1,  p- s- 2}(\overline{\mathbb{F}}_p)\oplus  V^{r+1, 0}_{p - r -2, s- 1}(\overline{\mathbb{F}}_p).$
From the above lemma we know that  $ {\rm  H}^0( \Gamma^1_{1, [\mathfrak{b}_i] }(\mathfrak{n}), V^{r, s}_{p- r- 1, p-s- 1}(\overline{\mathbb{F}}_p)) $ is non zero only when $ r = s = p-1$.  
Indeed $ V^{r, s}_{p-1-r, p-s-1}(\overline{\mathbb{F}}_p) \cong \overline{\mathbb{F}}_p$ as $\tilde{S}$-modules if and only if $ r = p-1, s = p-1.$
In this case we have $ (W^{ }_{r, s})^{\textrm{ss}} = V^{r, s}_{p- r- 1, p-s- 1}(\overline{\mathbb{F}}_p)\cong \overline{\mathbb{F}}_p.$ 
Therefore we obtain that
$$ {\rm H}^0(\Gamma^1_{1, [\mathfrak{b}_i]}(\mathfrak{n}), (W^{ }_{r, s})^{\textrm{ss}}) = \overline{\mathbb{F}}_p.$$ 
From the same lemma $ V^{ 0, s+1}_{r-1,  p- s- 2}(\overline{\mathbb{F}}_p)$ has non zero invariants only when $ r = 1,  s = p-2.$  
In this case we have  $$(W^{ }_{1, p-2})^{\textrm{ss}} = V^{1, p-2}_{p- 2, 1}(\overline{\mathbb{F}}_p) \oplus V^{ 0, p-1}_{0,  0}(\overline{\mathbb{F}}_p)\oplus  V^{2, 0}_{p -3, p-3}(\overline{\mathbb{F}}_p).$$
From this one has 
$$ {\rm H}^0(\Gamma^1_{1, [\mathfrak{b}_i]}(\mathfrak{n}), (W^{ }_{1, p-2})^{\textrm{ss}}) = \overline{\mathbb{F}}_p. $$
From the same lemma  the invariants of $ V^{r+ 1, 0}_{p-r-2, s-1}(\overline{\mathbb{F}}_p)$ are non zero if and only if  $ r = p-2, s  = 1.$ 
Similarly we obtain that $$ {\rm H}^0(\Gamma^1_{1, [\mathfrak{b}_i]}(\mathfrak{n}), (W^{ }_{p-2, 1 })^{\textrm{ss}}) = \overline{\mathbb{F}}_p. $$
As for the second item,  when $ r =  s = p-1,$ 
then $(W^{ }_{r, s})^{\textrm{ss}} = W^{ }_{ r, s}.$ Otherwise,  from the fact that  $ {\rm H}^0( \Gamma^1_{1, [\mathfrak{b}_i]}(\mathfrak{n}), (W^{}_{r, s})^{ss}) = 0,$  one deduces  that $ {\rm H }^0( \Gamma^1_{1, [\mathfrak{b}_i]}(\mathfrak{n}), W^{ }_{ r, s}   ).$\\
The remaining items are proved in a similar fashion.
\end{proof}
In the cases   $(r = 1,  s = p-2 \,\,\,\text{ and}\,\,\, f\mid e_1)$ or  $ (r  = p-2, s = 1\,\,\,\text{ and}\,\,f\mid e_2 ),$ further analysis is needed.

We  shall next discuss the case $r = 1, s = p-2 \,\,\text{ and}\,\,f\mid e_1$ in detail as it is symmetric to the remaining one.  We suppose in addition that $ p> 5.$ So  the  representation $ V^{l, t}_{1, p-2}(\overline{\mathbb{F}}_p)$ has dimension $ 2( p-1)$ and we identify it with its image in $ {\rm U}^{l+pt}_{ (p-1)^2}(\overline{\mathbb{F}}_p).$ Inside  $ {\rm U}^{l+pt}_{ (p-1)^2}(\overline{\mathbb{F}}_p)$ lies the submodule  $ M$ generated by the homogeneous monomials of degree $ ( p-1)^2.$  The dimension of $ M$ is  $  (p-1)^2 +1$ and it contains $ V^{l, t }_{ 1, p-2}( \overline{\mathbb{F}}_p)$ as submodule. By dimensional consideration ( it is here that we need to have $ p > 5$ to avoid discussing many cases), one deduces  an exact sequence  of  $ \overline{\mathbb{F}}_p[\tilde{G}]$-modules
$$  0 \rightarrow  V^{l, t}_{1, p-2}(\overline{\mathbb{F}}_p) \rightarrow M  \rightarrow  V^{2+l, t}_{p-3, p-3}(\overline{\mathbb{F}}_p) \rightarrow 0. $$  
Indeed from 
\begin{eqnarray*} 
(W^{l, t}_{ 1, p-2})^{\textrm{ss}}   =   V^{l, p-1 + t}_{ 0, 0}(\overline{\mathbb{F}}_p) \oplus V^{2+l, t}_{ p-3, p-3 }(\overline{\mathbb{F}}_p) \oplus V^{1+l, p-2+t}_{ p-2, 1}(\overline{\mathbb{F}}_p), 
\end{eqnarray*}
we know that the constituents of any  submodule of   $W^{l, t}_{ 1, p-2}$  are among the representations  $V^{l, p-1}_{ 0, 0}(\overline{\mathbb{F}}_p), \\  V^{2+l, t}_{ p-3, p-3 }(\overline{\mathbb{F}}_p)$ and  $V^{1+l, p-2+t}_{ p-2, 1}(\overline{\mathbb{F}}_p).$  From the equality $ (p-1 )^2 + 1 = (p-2)^2 +2(p -1),$ it follows that $$ M/V^{l, t}_{ 1, p-2}(\overline{\mathbb{F}}_p) \cong V^{2+l, t }_{ p-3, p-3}(\overline{\mathbb{F}}_p).$$ 
Therefore  $ V^{2+l, t }_{ p-3, p-3}(\overline{\mathbb{F}}_p)$ is a submodule of  $ {\rm U}^{l+pt}_{ (p-1)^2} / V^{l, t}_{ 1, p-2}(\overline{\mathbb{F}}_p) = (W^{l, t}_{ 1, p-2})^{\textrm{ss}}.$\\
Next we can realize the  module $ V^{l, p-1+t}_{0, 0}(\overline{\mathbb{F}}_p)$ as submodule of $   (W^{l, t}_{ 1, p-2})^{\textrm{ss}}$ by sending $ 1$ to the class $  X^{p(p-1)  }Y^{ p(p-1)} + V^{l , t}_{1, p-2}(\overline{\mathbb{  F}}_p).$  To see this we define 
\begin{eqnarray*} \varphi :   V^{l, p-1+t}_{0, 0}(\overline{\mathbb{F}}_p)  & \rightarrow & {\rm U}^{l +pt}_{ (p-1)^2}(\overline{\mathbb{F}}_p) /V^{l, t}_{ 1, p-2}(\overline{\mathbb{F}}_p) \\ 1 &\mapsto&  X^{p(p-1)  }Y^{ p(p-1)} + V^{l, t}_{1, p-2}(\overline{\mathbb{F}}_p).                                                                \end{eqnarray*}
Then for  $  g = \left(\begin{smallmatrix}  a & b \\ c & e  \end{smallmatrix}\right)
\in \tilde{G},$ we need to check that  $ \varphi(1) = g.\varphi(1).$ We have  $$ g.X^{ p( p-1)}Y^{ p( p-1)} = (ae - bc)^{l +pt}(a^p X^p + b^pY^p)^{p-1}(c^pX^p + e^pY^p)^{p-1}.$$ The latter polynomial is a linear combination of the  monomials  $  X^{ 2p^2 -2p - i}Y^i$  with  $ p|i.$ For all multiples $ i$ of $p$ less or equal to $ 2 p(p-1)$ except $ p(p-1)$ the monomials $  X^{ 2p^2 -2p - i}Y^i$ belong to  $ V^{l, t}_{1, p-2}(\overline{\mathbb{F}}_p).$ Indeed  let $ i = pk $, we recall the relations $ X^{ p^2} = X, Y^{ p^2} = Y $ in $ {\rm U}^{l +pt}_{ (p-1)^2}(\overline{\mathbb{F}}_p),$ then we have $$  X^{ 2p^2 - 2p  - pk}Y^{ pk} = X^{  p^2 -  2p - pk}X^{ p^2}Y^{ pk} = X^{( p-1)^2 - pk} Y^{ pk }   \in  V^{l, t}_{1, p-2}(\overline{\mathbb{F}}_p).$$ Therefore  $ g.\varphi( 1) \equiv \varphi(1) \pmod{ V^{l, t}_{r, s}(\overline{\mathbb{F}}_p)}.$
This implies that the direct sum $ V^{l, p-1}_{0, 0}(\overline{\mathbb{F}}_p)\oplus V^{1+l, t}_{ p-3, p-3}(\overline{\mathbb{F}}_p)$ is a submodule of $W^{l, t}_{1, p-2}.$  Thus we get an exact sequence $$  0\rightarrow V^{ l, p-1 +t}_{0, 0}(\overline{\mathbb{F}}_p)\oplus V^{2+l, t}_{ p-3, p-3}(\overline{\mathbb{F}}_p) \rightarrow  (W^{l, t}_{ 1, p-2}) \rightarrow V^{1+l, p-2+t}_{p-2, 1}(\overline{\mathbb{F}}_p)\rightarrow 0. $$ 
Hence for  $ ( r = 1, s = p-2 \,\,\text{and} \,\, l +pt \equiv p-1 \pmod{p^2-1}),$  we obtain that $$ {\rm H }^0( \Gamma_{1,[\mathfrak{b}_i]}(\mathfrak{ n }), W^{l, t}_{r, s})  =  \overline{\mathbb{F}}_p.$$ 
For $ (r = p-2, s = 1)$ and $ l +pt \equiv  1-p \pmod{p^2-1},$  similar arguments yield  $$ {\rm H }^0( \Gamma_{1, [\mathfrak{b}_i]}(\mathfrak{ n }), W^{l,  t}_{r, s})  =  \overline{\mathbb{ F}}_p.$$
In summary we have the following 
\begin{lemma}\label{lem2_3_7} Let $ p > 5$ and $ e := l +pt.$ Then  we have
\begin{enumerate}
\item $ {\rm H }^0( \Gamma^1_{1, [\mathfrak{b}_i]}(\mathfrak{ n}), W^{}_{r, s})   = 
\overline{\mathbb{F}}_p \,\,\,\,\text{if} \begin{cases} r = 1, s = p-2 \,\,\text{or} \,\, \\
r = p -2,  s = 1
\end{cases}$
\item $ {\rm H }^0( \Gamma_{1, [\mathfrak{b}_i]}(\mathfrak{n}), W^{l, t}_{r, s})   = 
\overline{\mathbb{F}}_p \,\,\,\text{if}\,\,\begin{cases} r = 1, s = p-2 \,\,\text{and}\,\,f\mid p (p-1)+e \,\ \text{or} \\
r = p -2,  s = 1\,\,\text{and}\,\,f\mid e + p-1.
\end{cases}$
\end{enumerate}
\end{lemma}
\subsubsection{Some indexes}
Let $\mathfrak{t}$ be  a finite place coprime with $ \mathfrak{pn}$. The matrix  $g\in Mat_2( \hat{\mathcal{O}} ) $ which has at the $ \mathfrak{ t}$-th place the matrix $\left(\begin{smallmatrix} \pi_\mathfrak{t} & 0 \\ 0 & 1
\end{smallmatrix}\right)$ where $ \pi_\mathfrak{t}$ is a uniformizer of $ \mathcal{O}_\mathfrak{t}$ and in all the remaining places has the identity matrix, belongs to $ \Delta^{\mathfrak{t}}_1(\mathfrak{n}).$
We shall fix in this sub-section such a  $ g.$
Because
$$ \left(\begin{smallmatrix} \pi^{-1}_\mathfrak{t} &  0 \\ 0 &  1 \end{smallmatrix}\right)\left(\begin{smallmatrix} a &  b \\ c &  d \end{smallmatrix}\right)\left(\begin{smallmatrix} \pi_\mathfrak{t} &  0 \\ 0 &  1 \end{smallmatrix}\right) = \left(\begin{smallmatrix} a &  \pi^{-1}_{\mathfrak{t}}b \\ \pi_{\mathfrak{t}}c &  d \end{smallmatrix}\right);
$$ we deduce that $$ K'_{1, g^{-1}}( \mathfrak{n} ) := g^{-1}K_1(\mathfrak{n}) g\cap K_1(\mathfrak{n}) =\left\{\left(\begin{smallmatrix} a &  b \\ c &  d \end{smallmatrix}\right) \in K_1(\mathfrak{ n}): \pi_{\mathfrak{t}}\mid c_{\mathfrak{t}} \right\}.$$
Consider also the subgroup $$ K^1_1( \mathfrak{ n} ) = \{ \alpha \in K_1(\mathfrak{n}): det(\alpha) = 1\}.$$
Similarly as in Lemma \ref{lem1}, one can prove that reduction modulo $ \mathfrak{t}$ provides us with a surjective homomorphism
$$  K^1_1(\mathfrak{n})  \twoheadrightarrow {\rm SL}_2( \mathcal{O}/\mathfrak{t}).$$
From this we deduce that we have a surjective  map
\begin{eqnarray*}
K_1(\mathfrak{n})  & \twoheadrightarrow &  \mathbb{P}^1( \mathcal{O}/\mathfrak{t}) \\
\left(\begin{smallmatrix} a  & b \\ c & d
\end{smallmatrix}\right) & \mapsto& (c:d).
\end{eqnarray*}
This map is surjective because there is a surjective map $ {\rm SL}_2( \mathcal{O}/\mathfrak{t}) \twoheadrightarrow \mathbb{P}^1(\mathcal{O}/\mathfrak{t}); \left(\begin{smallmatrix} a  & b \\ c & d
\end{smallmatrix}\right)  \mapsto (c:d) $ and also a  surjective map $ K_1( \mathfrak{n})\supset K^1_1( \mathfrak{n}) \twoheadrightarrow  {\rm SL}_2(\mathcal{O}/\mathfrak{t}).$
Since the subgroup $ K'_{1, g^{-1}}( \mathfrak{n} )$ is the subset of all elements that are mapped to $  ( 0: 1),$ we deduce that we have a bijection 
$$  K'_{1, g^{-1}}( \mathfrak{n} )\backslash K_1(\mathfrak{n}) \longleftrightarrow \mathbb{P}^1(\mathcal{O}/\mathfrak{t}).$$
Therefore we obtain  the index $ [K_1(\mathfrak{n}): K'_{1, g^{-1}}( \mathfrak{n} )] = N(\mathfrak{t}) +1.$  
Recall the definition  $  \Gamma'^{, \beta^{-1}_i}_{1, [\mathfrak{b}_i]}(\mathfrak{n}) := \Gamma_{1, [\mathfrak{b}_i]}(\mathfrak{n})\cap \beta^{-1}_i\Gamma_{1, [ \mathfrak{t}^{-1}\mathfrak{b}_i]}(\mathfrak{n})\beta_i. $ 
Similarly as $Y_{K_1(\mathfrak{n})} $ decomposes into disjoint union of its  connected component $ \amalg^h_{ i = 1 } \Gamma_{ 1, [ \mathfrak{b}_i]}(\mathfrak{n})\backslash \mathbb{H}_3, $   $ Y_{ K'_{1, g^{-1}}( \mathfrak{n})}$ decomposes as follows. We have  $ Y_{ K'_{1, g^{-1}}( \mathfrak{n})} = \amalg^h_{ i = 1 } \Gamma'^{, \beta^{-1}_i}_{ 1, [ \mathfrak{b}_i]}(\mathfrak{n})\backslash \mathbb{H}_3.$ Indeed we know that the connected components of $Y_{ K'_{1, g^{-1}}( \mathfrak{n})} $ are $   \Gamma''_{ 1, [ \mathfrak{b}_i]}(\mathfrak{n})\backslash \mathbb{H}_3$ where $ \Gamma''_{ 1, [ \mathfrak{b}_i]}(\mathfrak{n}) =  g_i K'_{1, g^{-1}}( \mathfrak{n}) g^{-1}_i  \cap G(F) = g_i g^{-1} K_1( \mathfrak{n}) g g^{-1}_i  \cap \Gamma_{1, [\mathfrak{b}_i]}(\mathfrak{n}).$ Recall that $ \beta_i = g_{j_i} g g^{-1}_i k_i = \left(\begin{smallmatrix} y_i & 0  \\
0 & 1\end{smallmatrix}\right)
 \in  G( F)$ with $ k_i = \left(\begin{smallmatrix} u_i & 0 \\
0 & 1\end{smallmatrix}\right)
\in  g_i K_1(\mathfrak{n})g^{-1}_i.$ 
For $ \sigma \in  \Gamma_{ 1, [\mathfrak{b}_{j_i}]}(\mathfrak{n}) =   g_{ j_i} K_1(\mathfrak{n}) g^{-1}_{ j_i} \cap G( F) $ we have
\begin{eqnarray*}
\beta^{-1}_i\sigma \beta_i  & \in & \beta^{-1}_i g_{j_i} K_1(\mathfrak{n}) g^{-1}_{j_i} \beta_i\cap G( F) \\
&= &  k^{-1}_i g_i g^{-1} g^{-1}_{ j_i}g_{ j_i} K_{1}(\mathfrak{n})g_{j_i}g^{-1}_{j_i} g g^{-1}_i k_i \cap G(F) \\
 & = &  g_i g^{-1} K_1(\mathfrak{n}) g g^{-1}_i \cap G(F)
\end{eqnarray*}
where we have used the facts that  $ g_i, g, k_i$ commute and $ k_i, k^{-1}_i \in g_i K_1(\mathfrak{n})g^{-1}_i.$
This means that  $ \beta^{-1}_i \Gamma_{ 1, [\mathfrak{b}_{ j_i}]}(\mathfrak{n}) \beta_i =  \Gamma''_{ 1, [\mathfrak{ b}_i]}(\mathfrak{n}).$
Therefore we deduce that
$$ \Gamma'^{, \beta^{-1}_i}_{1, [\mathfrak{b}_i]}(\mathfrak{n}) = \Gamma_{1, [\mathfrak{b}_i]}(\mathfrak{n})\cap \beta^{-1}_i\Gamma_{1, [ \mathfrak{b}_{j_i}]}(\mathfrak{n})\beta_i =  \Gamma''_{ 1, [\mathfrak{ b}_i]}(\mathfrak{n}).$$
So we have the following projection map
$$\begin{CD}
Y_{ K'_{1, g^{-1}}( \mathfrak{n})} =  \amalg^h_{ i = 1 } \Gamma'^{, \beta^{-1}_i}_{1, [ \mathfrak{b}_i]}(\mathfrak{n})\backslash \mathbb{H}_3 \\ 
@VVs_gV \\   
Y_{K_1(\mathfrak{n})} = \amalg^h_{ i = 1 } \Gamma_{ 1, [ \mathfrak{b}_i]}(\mathfrak{n})\backslash \mathbb{H}_3.  @. 
\end{CD}
$$ 
The map $ s_g$ is of degree $  [ K_1(\mathfrak{n}): K'_{1, g^{-1}}(\mathfrak{n})] = N(\mathfrak{ t}) + 1.$  
The maps induced by  $ s_g$  on the connected components are also of  degree $ N(\mathfrak{t}) +1.$
The discussion we just made  implies that for $ \beta_i$ corresponding to  $ g$, that is to mean $ K_1(\mathfrak{n}) g K_1(\mathfrak{n}) = K_1(\mathfrak{n})g^{-1}_{j_i}\beta_i g_{i}K_1(\mathfrak{n})$ where $ j_i$ is the unique index such  that  the ideal $(det( g_{j_i} g g^{-1}_{i}))$ is principal, the following holds.
\begin{lemma}\label{Index} 
Keeping the same assumptions as above, then for any ideal $\mathfrak{n}$ coprime with $ \mathfrak{t} = (det(g))$, we have
$$[{\Gamma_{1, [\mathfrak{b}_{j_i}]}(\mathfrak{n})}: \beta_{i}{\Gamma_{1, [\mathfrak{b}_{i}]}(\mathfrak{n})}\beta^{ -1}_i\cap {\Gamma_{1, [\mathfrak{b}_{j_i}]}(\mathfrak{n})}]= N(\mathfrak{t}) +1.$$
\end{lemma}
Therefore, the Hecke eigenvalue corresponding to the action of $ T_\mathfrak{t}$ on the $\overline{\mathbb{F}}_p$-vector space $ {\rm H}^0( \Gamma_{1, [\mathfrak{b}_i]}(\mathfrak{n}), V^{l, t}_{r, s}(\overline{\mathbb{F}}_p ))$ where  $ \mathfrak{t}$ is a prime ideal coprime with $ \mathfrak{pn}$ is  $ N(\mathfrak{t})+1.$ Hence eigenvalue systems coming from the $\overline{\mathbb{F}}_p$-vector space  $\oplus^h_{ i = 1}{\rm H}^0( \Gamma_{1, [\mathfrak{b}_i]}(\mathfrak{n}), V^{l, t}_{r, s}(\overline{\mathbb{F}}_p ))$ are Eisenstein\index{Eisenstein} because the semisimplification of the attached Galois representations is the direct sum of the cyclotomic character and the trivial character.


As we shall make use  of Shapiro's isomorphism, we need to verify that it is compatible with the Hecke action on group cohomology. From the above discussion,  we  deduce that we can choose identical  coset representatives for the double cosets $$ \Gamma_{1, [ \mathfrak{t}^{-1}\mathfrak{b}_i]}(\mathfrak{pn})\beta_i \Gamma_{1, [\mathfrak{b}_i]}(\mathfrak{pn})/\Gamma_{1, [\mathfrak{b}_i]}(\mathfrak{pn})\,\, \text{and for}\,\,  \Gamma_{1, [\mathfrak{t}^{-1}\mathfrak{b}_i]}( \mathfrak{n})\beta_i \Gamma_{1, [\mathfrak{b}_i]}(\mathfrak{n})/\Gamma_{1, [\mathfrak{b}_i]}(\mathfrak{n}).$$ This will be used for the compatibility of the  Hecke action with the Shapiro's isomorphism. 
\subsubsection{Compatibility of Shapiro's lemma with the Hecke action}
Recall that when $ \Gamma' < \Gamma$ are  congruence subgroups  and  $ M$ is a $\Gamma'$-module then  Shapiro's isomorphism\index{Shapiro's isomorphism} reads as $$ {\rm H }^*( \Gamma, {\rm Ind}^{ \Gamma}_{ \Gamma'} ( M)) \cong  {\rm H }^*( \Gamma',  M).$$ It is the isomorphism induced by  the restriction $j:  \Gamma' \hookrightarrow \Gamma $ and the homomorphism 
\begin{eqnarray*}
\phi:   {\rm Ind}^{ \Gamma}_{ \Gamma'} ( M) &\rightarrow& M \\
f  &\mapsto& f(1).                                                                                                   \end{eqnarray*}
Therefore in terms of cocycles we have
\begin{eqnarray*}
Sh: {\rm H }^*( \Gamma, {\rm Ind}^{ \Gamma}_{ \Gamma'} ( M)) &\rightarrow&  {\rm H }^*( \Gamma',  M) \\
c &\mapsto& \phi\circ c \circ j.
\end{eqnarray*}
From Subsection \ref{SubsecH}, we know that for $ g \in \Delta^\mathfrak{a}_1(\mathfrak{ pn})$ with $  \mathfrak{a}$ coprime with $\mathfrak{pn},$ any set of orbit representatives of the orbit space  $   K_1( \mathfrak{pn}) g K_1( \mathfrak{pn})/K_1( \mathfrak{pn})  $ belongs to  $ \Delta^\mathfrak{a}_1(\mathfrak{ pn})  $ where $ \mathfrak{a} = (det(g))\mathcal{O}.$ In a  similar fashion any set of orbit representatives of  $ K_1( \mathfrak{n}) g K_1( \mathfrak{n})/K_1( \mathfrak{n})$ belongs to $ \Delta^\mathfrak{a}_1(\mathfrak{n}).$ We also obtain that  any set of  representatives of the orbit space $\Gamma_{1, [\mathfrak{b}_{j_i}]}(\mathfrak{n})\beta_i \Gamma_{1,[\mathfrak{b}_{i}]}(\mathfrak{n})/\Gamma_{1,[\mathfrak{b}_{i}]}(\mathfrak{n})$ belong to  $ \Lambda^\mathfrak{c}_{1, [\mathfrak{b}_{i}]}(\mathfrak{n})$ when $ \beta_i$ is from $ \Lambda^\mathfrak{c}_{1, [\mathfrak{b}_{i}]}(\mathfrak{n}).$  For the forthcoming statement we need to recall some important facts. On page \pageref{Prop2-2}, we saw that reduction modulo $\mathfrak{p}$ provides us the following isomorphism of $ \Gamma_{1, [\mathfrak{b}_i]}( \mathfrak{n})$-modules:
$$ {\rm Ind}^{\Gamma_{1, [\mathfrak{b}_i]}(\mathfrak{n})}_{\Gamma_{1, [\mathfrak{b}_i]}(\mathfrak{pn})}(\overline{\mathbb{F}}_p)\cong {\rm Ind}^{\tilde{G}}_{\tilde{U}}(\overline{\mathbb{F}}_p) \cong {\rm Ind}^{\Gamma_{1, [\mathfrak{a}^{-1}\mathfrak{b}_i]}(\mathfrak{n})}_{\Gamma_{1, [\mathfrak{a}^{-1}\mathfrak{b}_i]}(\mathfrak{pn})}(\overline{\mathbb{F}}_p)   $$
where $ \tilde{U} = \{ \left(\begin{smallmatrix} a & b \\ 0 & 1 \end{smallmatrix}\right)
\in \tilde{G} \}\subset \tilde{B}$ with $ \tilde{B}$ the Borel subgroup of $ \tilde{G}$ and $ \overline{\mathbb{F}}_p $ is endowed with the structure of a trivial left $ \Gamma_{1, [\mathfrak{b}_i]}(\mathfrak{n})$-module.  The left action of the latter on  ${\rm Ind}^{\Gamma_{1, [\mathfrak{b}_i]}(\mathfrak{n})}_{\Gamma_{1, [\mathfrak{b}_i]}(\mathfrak{pn})}(\overline{\mathbb{F}}_p)$ is as follows: for $  \gamma \in \Gamma_{1, [\mathfrak{b}_i]}(\mathfrak{n})$ and $ f \in {\rm Ind}^{\Gamma_{1, [\mathfrak{b}_i]}(\mathfrak{n})}_{\Gamma_{1, [\mathfrak{b}_i]}(\mathfrak{pn})}(\overline{\mathbb{F}}_p)$ we have   $ (\gamma.f)(h):= f( h\gamma).$
By definition $${\rm Ind}^{\tilde{G}}_{\tilde{U}}(\overline{\mathbb{F}}_p) := \{ f  : \tilde{G}\rightarrow \overline{\mathbb{F}}_p: f( u h )= u f( h) = f(h), \, \forall \, u \in \tilde{U}, h \in \tilde{G}\},$$ that is the collection of all the  $\tilde{U}$-left invariant maps from $\tilde{G}$ to  $ \overline{\mathbb{F}}_p.$  Because for each  $ i = 1,  \cdots,  h, $  any element $\lambda \in \Lambda^\mathfrak{c}_{1, [\mathfrak{b}_i]}(\mathfrak{pn})$ has its reduction  belonging to $\tilde{U},$ we derive that any  $ f\in {\rm Ind}^{\tilde{G}}_{\tilde{U}}(\overline{\mathbb{F}}_p)$ satisfies$$ f(\lambda) = f(1).$$
\begin{proposition}
Let $ \mathfrak{a}$ be a prime ideal coprime with  $ \mathfrak{pn}.$ 
Let  $T_g = T_\mathfrak{a}$ be the Hecke operator associated with $g\in \Delta^\mathfrak{a}_1(\mathfrak{n})$ where $ \mathfrak{a} = det(g)\mathcal{O} $ the ideal corresponding to $ g.$ Explicitly $ g$ is the matrix  with the identity matrix in all finite places expect at the $ \mathfrak{a}$-place where there is the matrix $ \left(\begin{smallmatrix} \pi_\mathfrak{a} & 0 \\ 0 & 1\end{smallmatrix}\right).$ Here $ \pi_\mathfrak{a}$ is a uniformizer of $\mathcal{O}_\mathfrak{a}.$
Then the following diagram 
$$\begin{CD}
{\rm H}^1( \Gamma_{1, [\mathfrak{b}_i]}(\mathfrak{ n}), {\rm Ind}^{\Gamma_{1, [\mathfrak{b}_i]}(\mathfrak{n})}_{\Gamma_{1, [\mathfrak{b}_i]}(\mathfrak{pn})}(\overline{\mathbb{F}}_p) )  @> Sh>>   {\rm H}^1(\Gamma_{1, [\mathfrak{b}_i]}(\mathfrak{ pn}), \overline{\mathbb{F}}_p)    \\
@VVT_{\beta_i} V   @VVT_{\beta_i} V   \\
{\rm H}^1( \Gamma_{1, [\mathfrak{a}^{-1}\mathfrak{b}_i]}(\mathfrak{n}), {\rm Ind}^{\Gamma_{1, [\mathfrak{a}^{-1}\mathfrak{b}_i]}(\mathfrak{ n})}_{\Gamma_{1, [\mathfrak{a}^{-1}\mathfrak{b}_i]}(\mathfrak{pn})}(\overline{\mathbb{F}}_p))  @> Sh>> {\rm H}^1(\Gamma_{1, [\mathfrak{a}^{-1}\mathfrak{b}_i]}(\mathfrak{ pn}), \overline{\mathbb{F}}_p)
\end{CD}
$$
is well defined and is commutative.
\end{proposition}
\begin{proof}
The arrow that could  potentially not be  well defined  is the left vertical arrow. However this is  not an  issue since for each $i$ reduction modulo $p$ yields that the coefficients are isomorphic $ \Gamma_{1, [\mathfrak{b}_i]}(\mathfrak{n})$-modules as we just recalled above.
We next verify the commutativity of the diagram. We set  $ M :=  \overline{\mathbb{F}}_p.$ Let $ \beta_i$ corresponding to $g$ as provides by Lemma \ref{lem2_3_0}.
From a decomposition of the double coset $\Gamma_{1, [\mathfrak{a}^{-1}\mathfrak{b}_{i}]}(\mathfrak{pn})\beta_i \Gamma_{1, [\mathfrak{b}_i]}(\mathfrak{pn}) = \amalg_r \delta_r\Gamma_{1, [\mathfrak{b}_i]}(\mathfrak{pn}),$ we know that  \begin{eqnarray*} T_{\beta_i}: {\rm H}^1(\Gamma_{1, [\mathfrak{b}_i]}(\mathfrak{pn}), M)  & \rightarrow & {\rm H}^1(\Gamma_{1, [\mathfrak{a}^{-1}\mathfrak{b}_i]}(\mathfrak{ pn}), M)\\   &c \mapsto & ( w \mapsto \sum_r \delta_r. c(\delta^{-1}_r w \delta_{s_r}))
\end{eqnarray*}
where $ s_r$ is the unique index such that $ \delta^{-1}_r w \delta_{s_r} \in \Gamma_{1, [\mathfrak{b}_i]}(\mathfrak{pn}).$ Therefore for  $c$ a cocycle from  ${\rm H}^1(\Gamma_{1, [\mathfrak{b}_i]}(\mathfrak{ n}), {\rm Ind}^{\Gamma_{1, [\mathfrak{b}_i]}(\mathfrak{ n})}_{\Gamma_{1, [\mathfrak{b}_i]}(\mathfrak{pn})}(M) )$ and $ w$ from  $ \Gamma_{1, [\mathfrak{b}_i]}(\mathfrak{pn})$ we have  $$    \big(T_{\beta_i} \circ Sh(c)\big)(w) = \sum_r \delta_r\big( c(\delta^{-1}_r w \delta_{s_r})(1)\big).$$
Let $ \Gamma_{1, [\mathfrak{a}^{-1}\mathfrak{b}_i]}(\mathfrak{n})\beta_i \Gamma_{1, [\mathfrak{b}_i]}(\mathfrak{n}) =\amalg_s\lambda_s \Gamma_{1, [\mathfrak{b}_i]}(\mathfrak{n}) $  so that 
\begin{eqnarray*} 
T_{\beta_i}: {\rm H}^1(\Gamma_{1, [\mathfrak{b}_i] }(\mathfrak{n}),{\rm Ind}^{\Gamma_{1, [\mathfrak{b}_i]}(\mathfrak{n})}_{\Gamma_{1, [\mathfrak{b}_i]}(\mathfrak{pn})}(M) )&\rightarrow& {\rm H}^1(\Gamma_{1, [\mathfrak{a}^{-1}\mathfrak{b}_i] }(\mathfrak{n}),{\rm Ind}^{\Gamma_{1, [\mathfrak{a}^{-1}\mathfrak{b}_i]}(\mathfrak{ n})}_{\Gamma_{1, [\mathfrak{a}^{-1}\mathfrak{b}_i]}(\mathfrak{pn})}(M))\\ c & \mapsto & ( w \mapsto \sum_s \lambda_s c(\lambda^{-1}_s w \lambda_{n_s})).
\end{eqnarray*}
Here $ n_s$ is the unique index such that $ \lambda^{-1}_s w \lambda_{n_s} \in \Gamma_{1, [\mathfrak{b}_i]}(\mathfrak{n})$ for $ w \in \Gamma_{1, [\mathfrak{a}^{-1}\mathfrak{b}_i]}(\mathfrak{n}).$
We also have \begin{eqnarray*} Sh: {\rm H}^1( \Gamma_{1, [\mathfrak{a}^{-1}\mathfrak{b}_i]}(\mathfrak{n}), {\rm Ind}^{\Gamma_{1, [\mathfrak{a}^{-1}\mathfrak{b}_i]}(\mathfrak{n})}_{\Gamma_{1, [\mathfrak{a}^{-1}\mathfrak{b}_i]}(\mathfrak{pn})}(M))  & \rightarrow & {\rm H}^1( \Gamma_{1,[\mathfrak{a}^{-1}\mathfrak{b}_i] }(\mathfrak{pn}), M)\\  c &\mapsto & ( w \mapsto c(w)(1)).
\end{eqnarray*}
Now a set of coset representatives of  $ \Gamma_{1, [\mathfrak{a}^{-1}\mathfrak{b}_i]}( \mathfrak{pn})\beta_i\Gamma_{1, [\mathfrak{b}_i]}(\mathfrak{pn})/\Gamma_{1, [\mathfrak{b}_i]}(\mathfrak{pn})$ can be chosen to be identical to    coset representatives of  $ \Gamma_{1, [\mathfrak{a}^{-1}\mathfrak{b}_i]}( \mathfrak{n})\beta_i\Gamma_{1, [\mathfrak{b}_i]}( \mathfrak{n})/\Gamma_{1, [\mathfrak{b}_i]}( \mathfrak{n}).$ So choose  $ \delta_r \in \Lambda^\mathfrak{c}_{1, [\mathfrak{b}_i]}(\mathfrak{pn})$ such that $   \Gamma_{1, [\mathfrak{a}^{-1}\mathfrak{b}_i]}( \mathfrak{pn})\beta_i\Gamma_{1, [\mathfrak{b}_i]}( \mathfrak{pn}) = \amalg^k_r \delta_r \Gamma_{1, [\mathfrak{b}_i]}( \mathfrak{pn}).$ Also take $ \lambda_1 = \delta_1, \cdots, \lambda_k = \delta_k $  all belonging to $ \Lambda^\mathfrak{c}_{1, [\mathfrak{b}_i]}(\mathfrak{n})$ such that  $\Gamma_{1, [\mathfrak{a}^{-1}\mathfrak{b}_i]}( \mathfrak{n})\beta_i\Gamma_{1, [\mathfrak{b}_i]}( \mathfrak{n}) = \amalg^k_r \lambda_r \Gamma_{1, [\mathfrak{b}_i]}( \mathfrak{n}).$
Let $ c$ be a  cocycle  from $ {\rm H}^1( \Gamma_{1, [\mathfrak{b}_i]}(\mathfrak{ n}), {\rm Ind}^{\Gamma_{1, [\mathfrak{b}_i]}(\mathfrak{ n})}_{\Gamma_{1, [\mathfrak{b}_i]}(\mathfrak{pn})}(M) ) $ and $ w \in \Gamma_{1, [\mathfrak{a}^{-1}\mathfrak{b}_i]}(\mathfrak{pn}).$ 
We have then  $$ \big(Sh \circ T_{\beta_i}(c)\big)(w ) = \sum^k_r \big(\lambda_r c( \lambda^{-1}_r w \lambda_{n_r})\big)(1).$$ 
Because for each $ i = 1, \cdots,  h,$ the action of $ \Lambda^\mathfrak{c}_{1,  [\mathfrak{b}_i]}(\mathfrak{n})$ on $\overline{\mathbb{F}}_p$ is trivial,  we obtain that  $$ \delta_r\big( c(\delta^{-1}_r w \delta_{n_r})(1)\big) = c(\delta^{-1}_r w \delta_{n_r})(1).$$ But also since  $\lambda_r$ reduces modulo  $\mathfrak{pn}$ to an element in $ \tilde{U}$ we have that  $$ \big(\lambda_r c( \lambda^{-1}_r w \lambda_{n_r})\big)(1) =  c( \lambda^{-1}_r w \lambda_{n_r})(\lambda_r) = c( \lambda^{-1}_r w \lambda_{n_r})(1).$$
Therefore we deduce that $$  \sum^k_r \lambda_r c( \lambda^{-1}_r w \lambda_{n_r})(1) = \sum^k_r \delta_r c(\delta^{-1}_r w \delta_{n_r})(1).$$
\end{proof}
One other important fact that tells us that we only have to  look at Serre weights\index{Serre! weights}, that is to mean  irreducible $ \overline{ \mathbb{F}}_p[\tilde{G}]$-modules   for the analysis of  Hecke eigenclasses is the following proposition.
\begin{proposition}\label{prop3}
Let  $ \mathfrak{n}$ be an integral ideal such that the positive generator of $ \mathfrak{n}\cap \mathbb{Z}$ is greater than $ 3.$ Consider the open compact subgroup $ K_1(\mathfrak{n})$ of level $ \mathfrak{n}.$
Let    $M$ be a finite dimensional $\overline{ \mathbb{F}}_p[\tilde{G}]$-module. Let  $\Psi$ be an eigenvalue system occurring in $ \oplus^h_{ i = 1}{\rm H}^1( \Gamma_{1, [\mathfrak{b}_i]}(\mathfrak{n}), M)$  and taking   values in $ \overline{\mathbb{F}}_p.$ Then, there exists $ W,$ an irreducible subquotient of  $M$ such that $\Psi $  also occurs in $ \oplus^h_{ i = 1}{\rm H}^1(\Gamma_{ 1, [\mathfrak{b}_i]}(\mathfrak{n}),  W ).$  
\end{proposition}
\begin{proof} 
Let  $ W$ be an irreducible submodule  $ M.$  Denote the quotient $ M/W $ as  $  N.$ Set $  K = K_1(\mathfrak{n}).$ This is  an open compact subgroup of $ G(\hat{\mathcal{O}})$ which is neat and surjects onto $ \hat{\mathcal{O}}^*$ via the determinant.
Write the following  exact sequence of locally constant sheaves on  $ Y_K$ associated with $ W,  M,  N $ respectively:
$$ 0 \rightarrow \tilde{W} \rightarrow \tilde{M} \rightarrow  \tilde{N} \rightarrow 0.$$ From this one obtains the exact sequence in cohomology:
$$  \cdots\rightarrow {\rm H}^1( Y_K,  \tilde{W}) \rightarrow {\rm H}^1( Y_K,  \tilde{M}) \rightarrow {\rm H}^1( Y_K,  \tilde{N})\rightarrow \cdots.$$
Let  $ s$ be a system of Hecke eigenvalues from  $  {\rm H}^1( Y_K,  \tilde{M}).$ If the image  of $s$   is zero, then $s$ occurs in  $ {\rm H}^1( Y_K,  \tilde{W}),$ and we are done.  Otherwise it is arisen from $ {\rm H}^1( Y_K,  \tilde{N}).$ We then replace  $\tilde{M}$ by  $\tilde{N}$ and repeat the argument.
\end{proof}
\subsubsection{Statements and proofs of the main results}
The statement about the reduction to weight two is as follows.
\begin{theorem}\label{RedThm}
Let $F$ be an imaginary quadratic field of class number $h$. Let $\mathfrak{n}$ be an integral ideal in $F$ and let $ p> 5$ be a rational prime which is inert in $F$ and coprime with $ \mathfrak{n}.$ Suppose  that the positive generator of $ \mathfrak{n}\cap \mathbb{Z}$ is greater than $ 3.$ Let $ 0 \leq   r, s \leq p-1$ and $ 0\leq  l, t \leq p-1,$ with $ l, t$ not both equal to $p-1.$  Let $ \psi$ be a {\it system of Hecke eigenvalues} in $ \oplus^h_{ i = 1}{\rm H}^1( \Gamma_{1, [\mathfrak{b}_i]}( \mathfrak{n}), V^{l, t}_{r, s}( \overline{\mathbb{F}}_p ) ).$ Then  
$ \psi$ occurs in $ \oplus^h_{ i = 1}{\rm H}^1( \Gamma_{1, [\mathfrak{b}_i]}( \mathfrak{pn}), \overline{\mathbb{F}}_{p}\otimes det^{l+pt})$ except possibly when $  ( r = 1, s = p-2 )$  or $  ( r = p-2, s = 1).$ In these potential exceptions, the obstruction is coming from Hecke eigenvalue systems which are Eisenstein.
\end{theorem}
\begin{proof} 
The proof is divided in two parts. Firstly, we show  that $$ \oplus^h_{ i = 1}{\rm H}^1( \Gamma^1_{1, [\mathfrak{b}_i]}( \mathfrak{n}), V^{ }_{r, s}( \overline{\mathbb{F}}_p ) ) \hookrightarrow \oplus^h_{ i = 1}{\rm H}^1( \Gamma^1_{1, [\mathfrak{b}_i]}( \mathfrak{pn}), \overline{\mathbb{F}}_{p})$$ as $ \overline{\mathbb{F}}_p$-vector spaces except in the exceptional cases named in the statement.  Secondly from this, we use  an inflation  restriction exact sequence and obtain an embedding of Hecke modules  $\oplus^h_{ i = 1}{\rm H}^1( \Gamma_{1, [\mathfrak{b}_i]}( \mathfrak{n}), V^{l, t}_{r, s}( \overline{\mathbb{F}}_p)) \hookrightarrow \oplus^h_{ i = 1}{\rm H}^1( \Gamma_{1, [\mathfrak{b}_i]}( \mathfrak{pn}), \overline{\mathbb{F}}_{p}\otimes det^{l+pt})$.\\
{\bf First part:}\\
The exact sequence \[ 0 \rightarrow V^{}_{r,s}(\overline{\mathbb{F}}_p) \rightarrow  {\rm U}^{}_{ r +ps}(\overline{\mathbb{F}}_p) \rightarrow W^{}_{r, s} \rightarrow 0\]  gives rise to the long exact sequence in cohomology
\begin{eqnarray*} 
0 &\rightarrow &  \oplus^h_{ i = 1}{\rm H }^0( \Gamma^1_{1 , [\mathfrak{b}_i]}(\mathfrak{n} ), V^{ }_{r, s}(\overline{\mathbb{F}}_p))\rightarrow \oplus^h_{ i = 1}{\rm H }^0( \Gamma^1_{1, [\mathfrak{b}_i]}(\mathfrak{ n}), {\rm U}^{ }_{r +ps} (\overline{\mathbb{F}}_p) ) \rightarrow \\ &\rightarrow &   \oplus^h_{ i = 1}{\rm H }^0(\Gamma^1_{1, [\mathfrak{b}_i]}(\mathfrak{n}), W^{}_{r, s})   \rightarrow   \oplus^h_{ i = 1}  {\rm H}^1( \Gamma^1_{1, [\mathfrak{b}_i]}( \mathfrak{ n}), V^{}_{r, s}(\overline{\mathbb{F}}_p) ) \rightarrow \\  &\rightarrow &\oplus^h_{ i = 1} {\rm H }^1( \Gamma^1_{1, [\mathfrak{b}_i]}( \mathfrak{n}), {\rm U}^{ }_{ r +ps}( \overline{\mathbb{F}}_p) ) \rightarrow\oplus^h_{ i = 1} {\rm H }^1( \Gamma^1_{1, [\mathfrak{b}_i]}(\mathfrak{ n}), W^{}_{ r, s} ) \rightarrow \cdots.
\end{eqnarray*}
This is an exact sequence of  $ \overline{\mathbb{F}}_p$-vector spaces.
If  $ r = s = p-1,$  from  Lemmas  \ref{lem2},  \ref{lem3}  and  \ref{lem3.6},   we get the exact sequence of $ \overline{\mathbb{F}}_p$-vector spaces for each $ i = 1,  \cdots,  h:$ \[
0  \rightarrow \overline{\mathbb{F}}_p \rightarrow \overline{\mathbb{F}}_p \rightarrow {\rm H }^1(\Gamma^1_{1, [\mathfrak{b}_i]} (\mathfrak{n} ), V^{}_{r , s}(\overline{\mathbb{F}}_p) ) \rightarrow {\rm H }^1( \Gamma^1_{1, [\mathfrak{b}_i]}( \mathfrak{ n}),  {\rm U }^{}_{r + ps}(\overline{\mathbb{F}}_p)) \rightarrow \cdots.\] 
This means that the third arrow is the null map and hence we have an injection $$ \oplus^h_{i = 1}{\rm  H}^1( \Gamma^1_{1, [\mathfrak{b}_i]}(\mathfrak{ n} ), V^{}_{r, s}(\overline{\mathbb{F}}_p)) \hookrightarrow\oplus^h_{i = 1} {\rm H }^1( \Gamma^1_{1, [\mathfrak{b}_i]}(\mathfrak{n}), {\rm U}^{}_{ r +ps}(\overline{\mathbb{F}}_p)).$$ 
From  Lemmas  \ref{lem2},  \ref{lem3}  and  \ref{lem3.6}, we see that when  $  ( r \neq 1 \,\,\text{or}\,\, s \neq p-2) $  and $  ( r \neq p-2  \,\,\text{or}\,\, s \neq 1),$   we have an exact sequence of  $ \overline{\mathbb{F}}_p$-vector spaces $$ 0 \rightarrow  {\rm H}^1( \Gamma^1_{1, [\mathfrak{b}_i]}( \mathfrak{ n}), V^{ }_{r, s}(\overline{\mathbb{F}}_p))  \rightarrow {\rm H }^1( \Gamma^1_{1 , [\mathfrak{b}_i]}(\mathfrak{ n}), {\rm U }^{ }_{ r +ps}( \overline{\mathbb{F}}_p)) \rightarrow \cdots.$$ 
Therefore in all cases this is an exact sequence of $ \overline{\mathbb{F}}_p$-vector spaces.
From Proposition \ref{Prop2-2},   we know that the representation  $  {\rm U}_{ r+ps } ( \overline{\mathbb{F}}_p )$ is a direct summand of $ {\rm Ind}^{\Gamma^1_{1, [\mathfrak{b}_i]}( \mathfrak{n} ) }_{ \Gamma^1_{1, [\mathfrak{b}_i]}(\mathfrak{pn})}(\overline{\mathbb{F}}_p).$
So, one has an embedding of $ \overline{\mathbb{F}}_p$-vector spaces \[ \oplus^h_{ i = 1}{\rm H }^1( \Gamma^1_{1, [\mathfrak{b}_i]}( \mathfrak{ n}),{\rm  V}^{  }_{r, s}( \overline{\mathbb{F}}_p)) \hookrightarrow \oplus^h_{ i = 1}{\rm H }^1 ( \Gamma^1_{1, [\mathfrak{b}_i]}(\mathfrak{ n}),  {\rm Ind}^{\Gamma_{1, [\mathfrak{b}_i]}( \mathfrak{n} ) }_{ \Gamma^1_{1, [\mathfrak{b}_i]}( \mathfrak{pn} )} ( \overline{\mathbb{F}}_p)).\]
By Shapiro's lemma, one  concludes that we have an injection of $ \overline{\mathbb{F}}_p$-vector spaces \[\alpha:  \oplus^h_{ i = 1}{\rm H }^1(  \Gamma^1_{1, [\mathfrak{b}_i]}(\mathfrak{ n} ), V^{}_{ r, s}(\overline{\mathbb{F}}_p)) \hookrightarrow \oplus^h_{ i = 1}{\rm H }^1( \Gamma^1_{1, [\mathfrak{b}_i]}( \mathfrak{pn})), \overline{\mathbb{F}}_p).\]
Lastly when $  ( r = 1, s = p-2) $  or $  ( r = p-2, s = 1),$  then from  Lemmas \ref{lem2},  \ref{lem3},     \ref{lem3.6} \\ and \ref{lem2_3_7}, we have the exact sequence of $ \overline{\mathbb{F}}_p$-vector spaces
$$ 0 \rightarrow \overline{\mathbb{F}}_p \rightarrow  {\rm H}^1( \Gamma^1_{1, [\mathfrak{b}_i]}( \mathfrak{ n}), V^{ }_{r, s}(\overline{\mathbb{F}}_p) )  \rightarrow {\rm H }^1( \Gamma^1_{1 , [\mathfrak{b}_i]}(\mathfrak{ n}), {\rm U }^{  }_{ r +ps}( \overline{\mathbb{F}}_p)) \rightarrow \cdots. $$
{\bf Second part:}\\
Consider the  inflation-restriction exact sequence
\begin{eqnarray*}  0 \rightarrow  {\rm H}^1(\Gamma_{1, [\mathfrak{b}_i]}(\mathfrak{n})/\Gamma^1_{1, [\mathfrak{b}_i]}(\mathfrak{n})  , (V^{l, t}_{r, s}(\overline{\mathbb{F}}_p))^{\Gamma^1_{1, [\mathfrak{b}_i]}(\mathfrak{n})})  \xrightarrow{infl} {\rm H}^1(\Gamma_{1, [\mathfrak{b}_i]}(\mathfrak{n}), V^{l, t}_{r, s}(\overline{\mathbb{F}}_p)) \xrightarrow{res} \\ \xrightarrow{res} {\rm H}^1(\Gamma^1_{1, [\mathfrak{b}_i]}(\mathfrak{n}), V^{l, t}_{r, s}(\overline{\mathbb{F}}_p))^{\Gamma_{1, [\mathfrak{b}_i]}(\mathfrak{n})/\Gamma^1_{1, [\mathfrak{b}_i]}(\mathfrak{n})} \rightarrow {\rm H}^2(\Gamma_{1, [\mathfrak{b}_i]}(\mathfrak{n})/\Gamma^1_{1, [\mathfrak{b}_i]}(\mathfrak{n}),V^{l, t}_{r, s}(\overline{\mathbb{F}}_p)^{\Gamma^1_{1, [\mathfrak{b}_i]}(\mathfrak{n})}).
\end{eqnarray*}
Because of the assumption concerning  $p$ we have that $$ {\rm H}^1(\Gamma_{1, [\mathfrak{b}_i]}(\mathfrak{n})/\Gamma^1_{1, [\mathfrak{b}_i]}(\mathfrak{n})  , (V^{l, t}_{r, s}(\overline{\mathbb{F}}_p))^{\Gamma^1_{1, [\mathfrak{b}_i]}(\mathfrak{n})})  = {\rm H}^2(\Gamma_{1, [\mathfrak{b}_i]}(\mathfrak{n})/\Gamma^1_{1, [\mathfrak{b}_i]}(\mathfrak{n}),V^{l, t}_{r, s}(\overline{\mathbb{F}}_p)^{\Gamma^1_{1, [\mathfrak{b}_i]}(\mathfrak{n})}) = 0.$$
Then we get the isomorphism of $ \overline{\mathbb{F}}_p$-vector spaces induced by the restriction map:
$$  {\rm H}^1(\Gamma_{1, [\mathfrak{b}_i]}(\mathfrak{n}), V^{l, t}_{r, s}(\overline{\mathbb{F}}_p)) \xrightarrow{\sim}  ({\rm H}^1(\Gamma^1_{1, [\mathfrak{b}_i]}(\mathfrak{n}), V^{0, 0}_{r, s}(\overline{\mathbb{F}}_p))\otimes_{\overline{\mathbb{F}}_p}det^{l+pt})^{\Gamma_{1, [\mathfrak{b}_i]}(\mathfrak{n})/\Gamma^1_{1, [\mathfrak{b}_i]}(\mathfrak{n})}$$
where we have used the isomorphism $${\rm H}^1(\Gamma^1_{1, [\mathfrak{b}_i]}(\mathfrak{n}), V^{l, t}_{r, s}(\overline{\mathbb{F}}_p)) \backsimeq{\rm H}^1(\Gamma^1_{1, [\mathfrak{b}_i]}(\mathfrak{n}), V^{0, 0}_{r, s}(\overline{\mathbb{F}}_p))\otimes_{\overline{\mathbb{F}}_p}det^{l+pt}.$$
Next notice that for all $ 1 \leq i \leq h,$ we have isomorphisms of abelian groups
$$ \Gamma_{1, [\mathfrak{b}_i]}(\mathfrak{n})/\Gamma^1_{1, [\mathfrak{b}_i]}(\mathfrak{n}) \cong \mathcal{O}^* \cong {\Gamma_{1, [\mathfrak{b}_i]}(\mathfrak{pn})/\Gamma^1_{1, [\mathfrak{b}_i]}(\mathfrak{pn})}.$$
From the first part,  when we are in  the situation $  ( r \neq 1 \,\,\text{or}\,\, s \neq p-2 )$  and $  ( r \neq p-2 \,\,\text{or}\,\, s \neq  1),$ then   there is an embedding of $ \overline{\mathbb{F}}_p$-vector spaces:
$$ {\rm H}^1( \Gamma^1_{1, [\mathfrak{b}_i]}(\mathfrak{n}), V_{ r, s}(\overline{\mathbb{F}}_p) ) \hookrightarrow {\rm H}^1( \Gamma^1_{ 1, [ \mathfrak{b}_i]}(\mathfrak{pn}), \overline{\mathbb{F}}_p).$$
When tensoring with $ det^{l+pt},$  we obtain the embedding
$${\rm H}^1( \Gamma^1_{1, [\mathfrak{b}_i]}(\mathfrak{n}), V^{l, t}_{ r, s}(\overline{\mathbb{F}}_p) ) \hookrightarrow {\rm H}^1( \Gamma^1_{ 1, [ \mathfrak{b}_i]}(\mathfrak{pn}), \overline{\mathbb{F}}_p\otimes det^{l+pt}).$$
We next take  $ \mathcal{O}^*$-invariants and we get
$$({\rm H}^1( \Gamma^1_{1, [\mathfrak{b}_i]}(\mathfrak{n}), V^{l, t}_{ r, s}(\overline{\mathbb{F}}_p) ))^{\mathcal{O}^*} \hookrightarrow ({\rm H}^1( \Gamma^1_{ 1, [ \mathfrak{b}_i]}(\mathfrak{pn}), \overline{\mathbb{F}}_p\otimes det^{l+pt}))^{\mathcal{O}^*}.$$
This and the isomorphism induced by the inflation restriction exact sequence  implies that
$${\rm H}^1( \Gamma_{1, [\mathfrak{b}_i]}(\mathfrak{n}), V^{l, t}_{ r, s}(\overline{\mathbb{F}}_p) ) \hookrightarrow ({\rm H}^1( \Gamma^1_{ 1, [ \mathfrak{b}_i]}(\mathfrak{pn}), \overline{\mathbb{F}}_p\otimes det^{l+pt}))^{\mathcal{O}^*}.$$
Using once more the inflation restriction exact sequence for the right hand of this embedding, we derive that 
$${\rm H}^1( \Gamma_{1, [\mathfrak{b}_i]}(\mathfrak{n}), V^{l, t}_{ r, s}(\overline{\mathbb{F}}_p) ) \hookrightarrow {\rm H}^1( \Gamma_{ 1, [ \mathfrak{b}_i]}(\mathfrak{pn}), \overline{\mathbb{F}}_p\otimes det^{l+pt}).$$
This natural map is compatible with the Hecke action, and so this is an injection of Hecke modules.\\
Now when  the cases   $  ( r = 1, s = p-2 )$  or $  ( r = p-2, s = 1)$ hold, then the first part  provides us with an exact sequence
$$  0 \rightarrow \overline{\mathbb{F}}_p \rightarrow  {\rm H}^1( \Gamma^1_{1,[\mathfrak{b}_i]}(\mathfrak{n}), V_{ r, s}(\overline{\mathbb{F}}_p) ) \rightarrow {\rm H}^1( \Gamma^1_{1 , [\mathfrak{b}_i]}(\mathfrak{pn}), \overline{\mathbb{F}}_p ).
$$
This implies that the following  sequences are exact:
$$ 0 \rightarrow \overline{\mathbb{F}}_p\otimes det^{l+pt} \rightarrow  {\rm H}^1( \Gamma^1_{1,[\mathfrak{b}_i]}(\mathfrak{n}), V^{l, t}_{ r, s}(\overline{\mathbb{F}}_p) ) \rightarrow {\rm H}^1( \Gamma^1_{1 , [\mathfrak{b}_i]}(\mathfrak{pn}), \overline{\mathbb{F}}_p\otimes det^{l+pt} ) $$
$$ \Rightarrow  0 \rightarrow (\overline{\mathbb{F}}_p\otimes det^{l+pt})^{\mathcal{O}^*} \rightarrow  {\rm H}^1( \Gamma_{1,[\mathfrak{b}_i]}(\mathfrak{n}), V^{l, t}_{ r, s}(\overline{\mathbb{F}}_p) ) \rightarrow {\rm H}^1( \Gamma_{1 , [\mathfrak{b}_i]}(\mathfrak{pn}), \overline{\mathbb{F}}_p\otimes det^{l+pt} ).$$
Then, when  $ (\overline{\mathbb{F}}_p\otimes det^{l+pt})^{\mathcal{O}^*} = 0, $ the embedding
$$ {\rm H}^1( \Gamma_{1,[\mathfrak{b}_i]}(\mathfrak{n}), V^{l, t}_{ r, s}(\overline{\mathbb{F}}_p) ) \hookrightarrow {\rm H}^1( \Gamma_{1 , [\mathfrak{b}_i]}(\mathfrak{pn}), \overline{\mathbb{F}}_p\otimes det^{l+pt} ) $$
holds. Otherwise,  we know that the obstruction is coming from  $ (\overline{\mathbb{F}}_p\otimes det^{l+pt})^{\mathcal{O}^*}$ and is hence Eisenstein as shown by Lemma \ref{Index}.
\end{proof}
Systems of Hecke eigenvalues arising from $ \overline{\mathbb{F}}_p\otimes det^e$ are Eisenstein\index{Eisenstein} and hence correspond to reducible Galois representations. Because of this, the statement about Serre's conjecture\index{Serre!'s conjecture} is not affected since  it only concerns irreducible mod $p$ Galois representations. Now the  statement related to Serre type questions is as follows.
\begin{proposition}\label{SerreA}
We keep the same conditions as in Theorem \ref{RedThm}.
A positive answer to question $(a)$ on  page \pageref{Serre'squestions} answers positively the question $( b) $ and the reciprocal also holds.
\end{proposition}
\begin{proof}
The part $ (b) \Rightarrow (a)$ is obtained as follows. By Shapiro's lemma the system is realized in 
$$ \oplus^h_{ i = 1}{\rm  H}^1( \Gamma_{1, [\mathfrak{b}_i]}( \mathfrak{ n}), {\rm Ind }^{\Gamma_{1, [\mathfrak{b}_i]}(\mathfrak{n})}_{ \Gamma_{1, [\mathfrak{b}_i]}(\mathfrak{ pn})} (\overline{\mathbb{ F}}_p)).$$ 
By Proposition \ref{prop3}, this system of Hecke eigenvalues  already appears in 
$$ \oplus^h_{ i =1}{\rm  H}^1( \Gamma_{1, [\mathfrak{b}_i]}( \mathfrak{ n}), M)$$  where $ M$ is a simple module from the Jordan-H\"{o}lder series of $ {\rm Ind }^{ \Gamma_{1, [\mathfrak{b}_i]}( \mathfrak{ n})}_{ \Gamma_{1, [\mathfrak{b}_i]}( \mathfrak{ pn})} (\overline{\mathbb{F}}_p).$  This module $M$  is a 
Serre weight\index{Serre! weight}.\\
The part $ (a)  \Rightarrow (b)$ follows from  Theorem \ref{RedThm}.
\end{proof}

\def\refname{{References}}


\begin{thebibliography}{99}
\bibitem{Ash-Stevens} A. Ash and G. Stevens, {\it Modular forms in characteristic l and special values of their L-function}, Duke Math. J 53, no 3 849-868.
\bibitem{Ash-Stevens1} A. Ash and G. Stevens, {\it Cohomology of arithmetic groups and congruences between systems of Hecke eigenvalues}, J. Reine Angew. Math. 365 (1986), 192–220.
\bibitem{Ash1} A. Ash, Darrin Doud, and David Pollack, {\it Galois representations with conjectural connections to arithmetic cohomology}, Duke Mathematical Journal, Vol. 112, No. 3, 2002.
\bibitem{AshW}A. Ash and W. Sinnott,  {\it An analogue of Serre's conjecture for Galois representations
and Hecke eigenclasses in the mod-p cohomology of ${\rm GL}(n;\mathbb{ Z}) $},
Duke Math. J. 105 (2000), 1-24.
\bibitem{Bygott} J. S. Bygott, {\it Modular forms and modular symbols over imaginary quadratic fields}, PhD thesis, University of Exeter, 1998.
\bibitem{Borel} A. Borel, {\it  Introduction aux groupes arithm\'{e}tiques}, Publications de l’Institut de
Math\'{e}matique de l’Universit\'{e} de Strasbourg, XV. Actualit\'{e}s Scientifiques et Industrielles, No.
1341. Hermann, Paris.
\bibitem{Brown} K. S. Brown, {\it Cohomology of groups}, Graduate Texts in Mathematics, Springer-Verlag New-York, 1982.
\bibitem{Diamond} Fred Diamond, {\it A correspondence between representations of local Galois groups and Lie-type groups}, Proceedings of the LMS Durham Symposium on L-functions and Galois Representations, 2004.
\bibitem{Edix} B. Edixhoven, C.  Khare, {\it Hasse invariant and group cohomology}, Documenta Math  8  (2003) 43-50.
\bibitem{Emerton} M. Emerton, {\it $p$-Adic families of modular forms}, S\'{e}minaire Bourbaki, 62\`{e}me anne\'{e}, 2009-2010, \textnumero 1013, (2009).
\bibitem{Fi} L. M. Figueiredo, {\it  Serre's conjecture for imaginary quadratic fields}, Compositio Mathematica. 118 ( 1999), No. 1, 103-122.
\bibitem{Haluk-Seify} Mehmet Haluk Seng\"{u}n and Seyfi T\"{u}rkelli, {\it Weight Reduction for mod l Bianchi Modular forms}, to appear in Journal of Number Theory, Volume 129, Issue 8, August 2009, Pages 2010-2019.
\bibitem{Taylor} R. Taylor, {\it On congruences between modular forms}, PhD Thesis, Princeton University 1988.
\bibitem{Shimura} G. Shimura, {\it The special values of the zeta functions associated with Hilbert modular forms}, Duke Mathematical Journals, Vol.45, \textnumero. 3, (1978),  637-679.
\bibitem{Urban} Eric Urban, {\it Formes automorphes cuspidales pour $\rm GL_2 $ sur un corps quadratique imaginaire. Valeurs sp\'{e}ciales de fonctions L et congruences}, Compositio Mathematica, tome 99, No. 3 ( 1995), 283-324.
\bibitem{Gabor} G. Wiese, {\it On the faithfulness of parabolic cohomology as  a Hecke module over a finite field}, J. Reine Angew. Math. ( 2007), 79-103.  
\end{thebibliography}
\end{document}